\documentclass{amsart}
\usepackage{amsmath, amscd, amssymb, amsthm}
\usepackage{bbm}
\usepackage{latexsym}
\usepackage{amsfonts}
\usepackage{graphicx}
\usepackage[all,cmtip]{xy}

\usepackage[
dvipdfm, colorlinks=true, linkcolor=black,breaklinks=true,
urlcolor=blue, citecolor=black
]{hyperref}%
\usepackage{geometry}
\newtheorem{theorem}{Theorem}%[section]
\newtheorem{lemma}{Lemma}
\newtheorem{corollary}[theorem]{Corollary}

\newtheorem*{definition}{Definition}

\newcommand{\be}{\begin{equation}}
\newcommand{\ee}{\end{equation}}
\newcommand{\bea}{\begin{eqnarray}}
\newcommand{\eea}{\end{eqnarray}}

\newcommand{\vs}{\vspace{0.5cm}}

\newcommand{\vsv}{\vspace{0.12cm}}

\def\XXint#1#2#3{{\setbox0=\hbox{$#1{#2#3}{\int}$ }
\vcenter{\hbox{$#2#3$ }}\kern-.6\wd0}}

\begin{document}

\title{On Bismut Flat Manifolds}
\author{Qingsong Wang}

\address{Qingsong Wang. CMS, Zhejiang University, Hangzhou, 310027, China and Department of Mathematics,
The Ohio State University, 231 West 18th Avenue, Columbus, OH 43210,
USA} \email{qswang@zju.edu.cn}

\author{Bo Yang} \thanks{Research partially supported by an AMS-Simons Travel
Grant}

\address{Bo Yang. Department of Mathematics, Rutgers University, 110 Frelinghuysen Road
Piscataway, NJ 08854, USA.} \email{{boyang@math.rutgers.edu}}
\author{Fangyang Zheng} \thanks{Research partially supported by a Simons Collaboration Grant}

%\thanks{Research partially supported by NSFC 11271320}

%\address{Fangyang Zheng. Department of Mathematics, Zhejiang Normal University,
%Jinhua, 321004, Zhejiang, China and Department of Mathematics,
%The Ohio State University, 231 West 18th Avenue, Columbus, OH 43210,
%USA}

\address{Fangyang Zheng. Department of Mathematics,
The Ohio State University, 231 West 18th Avenue, Columbus, OH 43210, USA}

\email{{zheng.31@osu.edu}}

\begin{abstract}
In this paper, we give a classification of all compact Hermitian manifolds with flat Bismut connection. We show that the torsion tensor of such a manifold must be parallel, thus the universal cover of such a manifold is a Lie group equipped with a bi-invariant metric and a compatible left invariant complex structure. In particular, isosceles Hopf surfaces are the only Bismut flat compact non-K\"ahler surfaces, while central Calabi-Eckmann threefolds are the only simply-connected compact Bismut flat threefolds.
\end{abstract}

\maketitle

\tableofcontents

\markleft{On Bismut Flat Manifolds}
\markright{On Bismut Flat Manifolds}

\section{Introduction}

In recent years, there has been much progress on the study of
Hermitian differential geometry. Examples include the work of
Jixiang Fu and others on non-K\"ahler Calabi-Yau manifolds and
balanced manifolds, the work of Bo Guan on fully non-linear PDE with
application to Hermitian manifolds, the work of Hermitian curvature flows by Streets and Tian,
the study of general Hermitian
geometry and various Ricci curvature tensors by Liu and Yang, and
the recent solution to the Gauduchon Conjecture by Sz\'ekelyhidi,
Tosatti, and Weinkove. We refer the readers to \cite{Strominger},
\cite{Fu}, \cite{Fu-Yau}, \cite{Fu-Li-Yau}, \cite{Fu-Wang-Wu},
\cite{Fu-Wang-Wu1}, \cite{Li-Yau}, \cite{Tseng-Yau}, \cite{Guan},
\cite{Guan-Li}, \cite{Guan-Li1}, \cite{Guan-Sun}, \cite{Streets-Tian}, \cite{Streets-Tian1}, \cite{Streets-Tian2}, \cite{Streets-Tian3}, \cite{Liu-Yang},
\cite{Liu-Yang1}, \cite{Liu-Yang2}, \cite{Tosatti},
\cite{Tosatti-Weinkove}, \cite{STW} and the references therein for
some recent progress in this area.

\vsv

Given a Hermitian manifold $(M^n,g)$, there are three well-known canonical connections associated with the metric, namely, the
Riemannian (or Levi-Civita) connection $\nabla$, the Chern (aka Hermitian)
connection $\nabla^c$, and the Bismut connection $\nabla^b$. In
 \cite{Bismut}, Bismut showed that on any Hermitian manifold, there exists
 a unique connection that is compatible with the metric $g$ and the almost
 complex structure $J$, and whose $(3,0)$ torsion tensor is skew-symmetric (see also \cite{Yano}).
 This canonical connection is known as the Bismut connection.

 \vsv

 Bismut connection has been playing an increasingly important role in the study of non-K\"ahler geometry in recent years. For
   instance, one version of the definitions of Calabi-Yau manifolds with
   torsion refers to compact non-K\"ahler Hermitian threefolds
   (with finite fundamental group) whose Bismut connection has
   $SU(3)$ holonomy. As another example, in \cite{Streets-Tian2}, Streets and Tian used Bismut connection to reinterpret their Hermitian curvature flow (\cite{Streets-Tian}, \cite{Streets-Tian1}) and exhibited a remarkable relationship to mathematical
physics. They showed that, up to gauge equivalence, the flow is the renormalization group
flow of a nonlinear sigma model with nonzero B-field. As a consequence they concluded that the flow is a
gradient flow and discovered an entropy functional.

\vsv

When $g$ is K\"ahler, all three connections coincide, and when $g$
is not K\"ahler, these three connections are mutually different,
even though any one of them completely determines the other two. It
is certainly very natural to study the curvature of these
connections. In particular, one could ask: what kind of Hermitian
manifolds are ``space forms" with respect to each connection? The
simplest case of course would be the everywhere zero curvature case,
namely, to classify the flat spaces.

\vsv

For the Chern connection, the classical result of Boothby
\cite{Boothby} in 1958 states that, if $(M^n,g)$ is a compact
Hermitian manifold whose Chern connection is flat, then the
universal cover of $M$ is (holomorphically isometric to) a complex
Lie group equipped with a left invariant Hermitian metric. In
particular, there are no compact simply-connected Hermitian
manifolds with zero Chern curvature. When $n\geq 3$, such manifolds
need not be K\"ahler, as the famous Iwasawa manifold illustrates. An
important subclass of these manifolds are the complex parallelizable
manifolds, studied by H.-C. Wang (\cite{Wang}) in 1954.

\vsv

For the Riemannian (Levi-Civita) connection, the question is essentially to
find all compatible complex structures on the flat torus
$T^{2n}_{\mathbb R}$, since for any compact flat Riemannian manifold $M$,
a finite unbranched cover of $M$ is a flat torus by the Bieberbach Theorem.
 For $n=2$, the classification theory for compact complex surfaces implies
 that $M$ must be a complex torus, but when $n\geq 3$, there are
non-K\"ahlerian complex structures living on the flat torus. In a
recent work \cite{KYZ}, we were able to determine all orthogonal
complex structures on flat $6$-tori, thus solving the $n=3$
case. The classification problem for Riemannian flat compact
Hermitian manifold in complex dimension $4$ or higher remains open
at this point.

\vsv

In light of the above results, it is natural to ask the following

\vsv
\vsv

\noindent {\bf Question.} {\em What kind of
compact Hermitian manifolds will be Bismut flat, namely, the
curvature of the Bismut connection is everywhere zero? }

\vsv
\vsv

As Riemannian manifolds, the structure of such spaces are well known, as the classical theory of Cartan and Schouten (\cite{CS}, \cite{CS1}) states that the existence of a flat metric connection with skew-symmetric torsion would imply that the space is a Lie group or $S^7$ (or their products). See also the work of J. Wolf \cite{Wolf}, \cite{Wolf1} and the very nice new treatment by Agricola and Friedrich \cite{AF}. So the point here is to understand the complex structures compatible with these Riemannian metrics. In general, it is a challenging task to understand the set of all possible complex structures compatible with a given Riemannian metric, or in Simon Salamon's term, all orthogonal complex structures (OCS), see \cite{Salamon}. For example, it is still unknown what is the set of all OCSs on a flat $8$-torus, as we mentioned above. This set is known to be quite large, as it contains all the non-K\"ahlerian warped complex structures given by Borisov, Salamon, and Viaclovsky \cite{BSV}.

\vsv

Back to our Bismut flat manifolds, first let us see some examples of
such spaces. It is not hard to see that any isosceles Hopf surface
is Bismut flat (see \S 2). Also, recall that the ``central"
Calabi-Eckmann threefold is the manifold $M=S^3\times S^3$, equipped
with the product of (constant multiples of) the standard metric,
with a compatible complex structure that is left invariant when $M$
is considered as the Lie group $SU(2)\times SU(2)$. This complex
structure is the one constructed by Samelson \cite{S} for
even-dimensional compact Lie groups, and belongs to the family of
complex structures constructed by Calabi and Eckmann \cite{CE} on
the product of odd dimensional spheres. It is easy to check that the
central Calabi-Eckmann threefolds are Bismut flat.  More generally,
Samelson showed \cite{S} (see also \S 5 of the work of Alexandrov
and Ivanov \cite{AI}) that any compact Lie group $G$ of even
dimension admits left invariant complex structures that are
compatible with a bi-invariant metric.  It is well known that such a
Hermitian manifold is Bismut flat (see for instance the work
\cite{AI} or \cite{Joyce}). For the sake of convenience, let us
introduce the following terminology:

\begin{definition}
A {\bf
 Samelson space} is a Hermitian manifold $(G', g, J)$, where $G'$ is a connected and simply-connected, even-dimensional Lie group, $g$ a bi-invariant metric on $G'$, and $J$ a left invariant complex structure on $G'$ that is compatible with $g$.
\end{definition}

By Milnor's Lemma (\cite{M}, Lemma 7.5), a simply-connected Lie group $G'$ with a bi-invariant metric must be  the product of a compact semisimple
Lie group with an additive vector group,  namely, $G'=G\times {\mathbb R}^k$, where $0\leq k\leq \dim G'$ and $G$ is compact semisimple.
Notice that $G'$ and the compact Lie group $G''=G \times T^k$ (where
$T^k$ is the torus) share the same Lie algebra, so when $G'$ is
even dimensional, $G''$ hence $G'$ admits left invariant complex
structures compatible with the bi-invariant metric.

\vsv

Let $\rho : {\mathbb Z}^k \rightarrow I(G)$ be a homomorphism from
the free abelian group of rank $k$ into the isometry group of $G$.
Then $\Gamma_{\rho } \cong {\mathbb Z}^k$ acts on $G\times {\mathbb
R}^k$ by $\gamma (x,y) = (\rho(\gamma )(x), y+\gamma )$ as
isometries, and it acts freely and properly discontinuously, so we
get a compact quotient $M_{\rho} =(G\times {\mathbb
R}^k)/\Gamma_{\rho } $.

\begin{definition}
Let $(G',g,J)$ be a Samelson space, where $G'=G\times {\mathbb
R}^k$. Let $\rho : {\mathbb Z}^k \rightarrow I(G)$ be a homomorphism
into the isometry group of $G$, and $M_{\rho}$ be the compact
quotient defined as above. If the complex structure of $G'$ is
preserved by $\Gamma_{\rho}$, then it descends down to $M_{\rho}$
and makes it a complex manifold. In this  case we will call the
compact Hermitian manifold $M_{\rho}$ a {\bf local Samelson space}.
Such a Hermitian manifold is Bismut flat, since its universal cover
is so.
\end{definition}

As we shall see in the proof of Theorem 1 below,  $M_{\rho}$ is always diffeomorphic to $G\times T^k$, where $T^k$ is the $k$-torus. However, $M_{\rho}$ (or any finite unbranched cover of it) may not be a Lie group. Also, $G$ (considered as left multiplications) is a proper
subgroup of $I(G)$ in general. When the image of $\rho$ is
contained in $G$, then $\Gamma_{\rho }$ acts as left multiplications
in $G'$ hence preserves the complex structure. In particular,
$M_0=G\times T^k$ is always a compact Bismut flat manifold.

\vsv

A somewhat surprising fact to us is that, compact Bismut flat
manifolds actually form a rather small class, and they are
essentially just these  local Samelson spaces. To be precise, we
have the following:

\begin{theorem} \label{main} Let $(M^n,g)$ be a compact Hermitian
manifold whose Bismut connection is flat. Then there exists a finite
unbranched cover $M'$ of $M$ such that $M'$ is a local Samelson
space $M_{\rho}$ defined as above. Also, $M_{\rho}$ is diffeomorphic
to $G\times T^k$, where $T^k$ is the $k$-torus.
\end{theorem}

We remark that the finite unbranched cover $M'\rightarrow M$ might not be Galois, even though there is always a finite sequence of Galois covers $M_{i+1} \rightarrow M_i$, $1\leq i \leq r-1$, such that $M'=M_r$ and $M_1=M$. See \S 4 for more details.

\vsv

Note that by a result of Pittie \cite{P} (see also \S 5 of
\cite{AI}), on an even dimensional Lie group $G'$ with bi-invariant
metric, any compatible left invariant complex structures must be
those constructed by Samelson, from the root decompositions. (See \S
3 for a more detailed discussion of this). We remark that while the
universal covering space $\widetilde{M}$ of a compact Bismut flat
manifold $M$ is always a Samelson space, in general, however, the
deck transformation group $\Gamma $ might not be a subgroup of the
Lie group $\widetilde{M}$ (see also the discussion in \S 4).

\vsv

A property about Bismut flat manifolds worth mentioning is the following. This and other properties about such manifolds are actually used in the proof of Theorem 1. Recall that (see \cite{Gauduchon}) the {\em Gauduchon $1$-form} $\eta $ of a Hermitian manifold $M^n$ is the global $(1,0)$ form on $M$ determined by  $\partial \omega^{n-1} = -2 \eta \wedge \omega^{n-1}$, where $\omega$ is the K\"ahler (metric) form.  The manifold is called balanced, if $d(\omega^{n-1})=0$, or equivalently, $\eta =0$.

\begin{theorem}
Let $(M^n,g)$ be a Hermitian manifold which is Bismut flat. If it is balanced, then it is K\"ahler. If $M$ is compact, then the equality
$$ \int_M |T^c|^2 \omega^n = 16 \int_M |\eta |^2 \omega^n $$
holds, where $T^c$ is the torsion of the Chern connection, and $\eta$ is the Gauduchon $1$-form which is the trace of the torsion.
\end{theorem}

Now let us consider the special case of $n=2$. Recall that an {\em isosceles Hopf surface} is a compact complex surface $M^2$ with universal cover ${\mathbb C}^2\setminus \{ 0\}$ and deck transformation group  a finite extension (by unitary rotations) of the infinite cyclic group ${\mathbb Z}f$ with $f(z_1,z_2)=(az_1, bz_2)$, where $(z_1, z_2)$ is the Euclidean coordinate of ${\mathbb C}^2$ and $0<|a|=|b|<1$.

\vsv

Write $|z|^2$ for $|z_1|^2+|z_2|^2$. The {\em standard} Hermitian
metric $g$ on ${\mathbb C}^2\setminus \{ 0\} $ has the K\"ahler form
$\omega_g = \frac{\sqrt{-1} }{|z|^2} \partial \overline{\partial }
|z|^2$. When $|a|=|b|$, this metric descends down to $M^2$.  It is
straight forward to check (see \S 2) that this metric is Bismut
flat. So as a direct consequence of Theorem 1, we get the following

\begin{corollary}\label{cor1}
Let $(M^2,g)$ be a compact Hermitian surface that is Bismut flat.
Then either $g$ is K\"ahler and $(M^2,g)$ is a flat complex torus or
a flat hyperelliptic surface, or $g$ is not K\"ahler and $(M^2,g)$
is an isosceles Hopf surface equipped with (a constant multiple of)
the standard metric.
\end{corollary}

\vsv

Next let us examine the $3$ dimensional cases. Recall that a {\em
central Calabi-Eckmann threefold} is the Lie group $G=SU(2)\times
SU(2)$ equipped with a left invariant complex structure $J$ which is
compatible with a bi-invariant metric. The bi-invariant metrics on
$G$ are unique up to scaling constants on the factors, namely, they
are in the form $c_1g_0 \times c_2g_0$, where $c_1$, $c_2$ are positive
constants and $g_0$ the standard metric on $SU(2)=S^3$ which has
constant sectional curvature $1$.  As Hermitian manifolds, the
holomorphic isometric class of such spaces are determined by the two
scaling constants.

\begin{corollary}\label{cor2}
Let $(M^3,g)$ be a compact Hermitian manifold of dimension $3$ which
is Bismut flat and non-K\"ahler. Then the universal cover of $M$ is
holomorphically isometric to either a central Calabi-Eckmann
threefold, or $({\mathbb C}^2\setminus \{ 0\})\times {\mathbb C}$,
equipped with the product of (a constant multiple of) the standard
metric $\omega_g$ and the flat metric.  In particular, the only
simply-connected compact Bismut flat threefolds are the central
Calabi-Eckmann threefolds.
\end{corollary}

Notice that the central Calabi-Eckmann manifold $S^3\times S^3$, being the only compact, simply-connected, three dimensional Bismut flat manifold, it just stands out as perhaps a perfect candidate for the hidden space of our universe in the non-K\"ahler Calabi-Yau theory.

\vsv

Also, in the case when $\widetilde{M}= ({\mathbb C}^2\setminus \{ 0\})\times {\mathbb C}$, there might not be any finite unbranched cover $M'$ of $M$ such that $M'$ is the product of an isosceles Hopf surface and an elliptic curve. In fact, $M'$ might not even be an elliptic fibration over a Hopf surface or a Hopf surface fibration over an elliptic curve. We will see an example of such kind at the end of \S 4.

\vsv

Using the fact that the only simply-connected compact simple Lie
groups in dimension less than $14$ are $SU(2)$, $SU(3)$, and
$Spin(5)$, one could have a similar discussion and classification on
compact Bismut flat manifolds in complex dimension $6$ or less.

\vsv

Finally, it is natural to wonder about the compactness assumption in
Theorem 1. By using a nice characterization of flat metric connection with skew-symmetric torsion on a Riemannian manifold, given by I. Agricola and T. Friedrich in \cite{AF}, we obtain the following generalization of the main theorem:

\begin{theorem} \label{noncompact} Let $(M^n,g)$ be a simply-connected  Hermitian
manifold whose Bismut connection is flat. Then there exists a
Samelson space $(G, J, g_0)$, namely, $G$ is a simply-connected
even-dimensional Lie group, with  $g_0$ a bi-invariant metric on $G$
and $J$ a compatible left invariant complex structure on $G$, such
that $M$ is an open complex submanifold of $G$ and $g=g_0|_{M}$.
\end{theorem}

In other words, the universal cover of any Bismut flat manifold is
always an open part of a Samelson space. In particular, any
simply-connected, non-K\"ahler, Bismut flat surface is an open
subset in $({\mathbb C}^2\setminus \{ 0\} , cg)$, where $c>0$ is a
constant and $\omega_g= \frac{ \sqrt{-1} }{ |z|^2 } \partial
\overline{\partial } |z|^2  $ is the standard metric.

\vsv

In comparison, for manifolds with flat Chern connections, Boothby
\cite{Boothby} observed that there are non-compact Chern flat
surfaces whose torsion components are not constants. In \S 5, we will give more examples of such kind, including a complete one.

\vsv

For Hermitian manifolds with flat Riemannian connection, there are also lots of non-compact examples with non-constant norm of torsion, even in complex dimension $2$. It turns out that locally such structures are determined by three holomorphic functions. We will give some examples of such surfaces in the end. All such examples are necessarily incomplete, since a complete flat Riemannian $4$-manifold is uniformized by the flat Euclidean space ${\mathbb R}^4$, and all orthogonal complex structures on ${\mathbb R}^4$ are the standard ones, by a result of Salamon and Viaclovsky (\cite{SV}, Theorem 1.3).

\vsv

The paper is organized as follows: In Section 2, we  collect some
preliminary results. In Section 3, we recall the construction of
Samelson and Pittie on left invariant complex structures on
even-dimensional compact Lie groups. In Section 4, we discuss the
general properties of Bismut flat manifolds, and give the proofs of
Theorem and the corollaries. In Section 5, we discuss the non-compact cases.

\vs

\vsv

\vsv

\vsv

\section{Preliminaries}

In this section, we will collect some preliminary results and fix the  notations and terminologies. More details can be found in our earlier work \cite{YZ}, but we will try to make things self-contained here for the convenience of the readers.

\vsv

Let $(M^n,g)$ be a Hermitian manifold, with $n\geq 2$. We will
denote by $\nabla$, $\nabla^c$, and $\nabla^b$ the Riemannian, Chern, and Bismut
connection of the metric $g$, and by $R$, $R^c$, and $R^b$  their curvatures, called
the Riemannian, Chern, or Bismut curvature tensor, respectively. (In \cite{YZ} we used the term Hermitian instead of Chern. The latter is a less ambiguous in this context).

\vsv

Let  $T^{1,0}M$ be the bundle of complex tangent vector fields of type $(1,0)$, namely,
complex vector fields of the form  $v- \sqrt{-1}Jv$, where $v$ is a real vector field on $M$. Let $\{ e_1, \ldots , e_n\}$ be a local frame of $T^{1,0}M$ in a
neighborhood in $M$. Write $e=\ ^t\!(e_1, \ldots , e_n) $
 as a column vector. Denote by $\varphi = \ ^t\!(\varphi_1, \ldots ,
  \varphi_n)$ the column vector of local $(1,0)$-forms which is the
coframe dual to $e$. For the Chern connection $\nabla^c$ of $g$,
let us denote by $\theta$,  $\Theta$ the matrices of connection and
curvature, respectively, and by $\tau$ the column vector of the
torsion $2$-forms, all under the local frame $e$. Then the structure
equations and Bianchi identities are
\begin{eqnarray}
d \varphi & = & - \ ^t\!\theta \wedge \varphi + \tau,  \label{formula 1}\\
d  \theta & = & \theta \wedge \theta + \Theta. \\
d \tau & = & - \ ^t\!\theta \wedge \tau + \ ^t\!\Theta \wedge \varphi, \label{formula 3} \\
d  \Theta & = & \theta \wedge \Theta - \Theta \wedge \theta.
\end{eqnarray}
Note that under a frame change $\tilde{e}=Pe$, the corresponding
forms are changed by
$$ \tilde{\varphi }= \ ^t\!P^{-1} \varphi , \ \ \tilde{\theta } =
P\theta P^{-1} + dPP^{-1}, \ \ \tilde{\Theta }= P\Theta P^{-1},
\ \ \tilde{\tau } = \ ^t\!P^{-1} \tau $$
In particular, the types of the $2$-forms in $\Theta $
and $\tau$ are independent of the choice of the frame $e$. Also, the compatibility of $\nabla^c$ with the metric means that when $e$ is unitary, both $\theta$ and $\Theta$ would be skew-Hermitian. Using these facts, and by taking $e$ to be either holomorphic or unitary, we know that the entries of $\Theta$ or $\tau $ are always $(1,1)$ or $(2,0)$ forms, under any frame.

\vsv

Let us write  $\langle \ , \rangle $ for the (real) inner product given by the Hermitian metric $g$, and extend it bilinearly over ${\mathbb C}$. Under the frame $e$, let us denote the components of the Riemannian connection $\nabla$ as
$$ \nabla e = \theta_1 e + \overline{\theta_2 }\overline{e} ,
\ \ \ \nabla \overline{e} = \theta_2 e + \overline{\theta_1
}\overline{e} ,$$ then the matrices of connection and curvature for
$\nabla $ become:
$$ \hat{\theta } = \left[ \begin{array}{ll} \theta_1 & \overline{\theta_2 } \\ \theta_2 & \overline{\theta_1 }  \end{array} \right] , \ \  \  \hat{\Theta } = \left[ \begin{array}{ll} \Theta_1 & \overline{\Theta}_2  \\ \Theta_2 & \overline{\Theta}_1   \end{array} \right] $$
 where
\begin{eqnarray}
\Theta_1 & = & d\theta_1 -\theta_1 \wedge \theta_1 -\overline{\theta_2} \wedge \theta_2 \\
\Theta_2 & = & d\theta_2 - \theta_2 \wedge \theta_1 - \overline{\theta_1 } \wedge \theta_2  \label{formula 7}\\
d\varphi & = & - \ ^t\! \theta_1 \wedge \varphi - \ ^t\! \theta_2
\wedge \overline{\varphi } ,
\end{eqnarray}
and under the frame change $\tilde{e}=Pe$,
$ \overline{\tilde{e}} = \overline{P}\overline{e}$,
the above matrices of forms are changed by
$$ \tilde{\theta}_1  = P\theta_1P^{-1} +dPP^{-1}, \ \ \tilde{\theta}_2  =
\overline{P} \theta_2 P^{-1}, \ \ \tilde{\Theta}_1 = P\Theta_1P^{-1},
\ \ \tilde{\Theta}_2 = \overline{P}\Theta_2P^{-1} $$
Following \cite{YZ}, we will write
\begin{equation}
\gamma = \theta_1 - \theta .
\end{equation}
We have $\tilde{\gamma } = P\gamma P^{-1}$ under the frame change, so $\gamma$ represents a tensor.
The compatibility of $\nabla$ with the metric implies that when $e$ is unitary,
both
$\theta_2 $ and $\Theta_2$ are skew-symmetric, while $\theta_1$,
$\gamma$, or $\Theta_1$ are skew-Hermitian.

\vsv

Let $\gamma = \gamma ' + \gamma ''$ be the decomposition
of $\gamma$ into $(1,0)$ and $(0,1)$ parts. Denote by $T_{ij}^k=-T_{ji}^k$ the components of $\tau$:
\begin{equation}
\tau_k = \sum_{i,j=1}^n T_{ij}^k \varphi_i\wedge \varphi_j \ = \sum_{1\leq i<j\leq n} 2 \ T_{ij}^k \varphi_i\wedge \varphi_j
\end{equation}
Note that our $T_{ij}^k$ is only half of the components of the torsion $\tau$ used in some other literature where the second sigma term is used. As observed in \cite{YZ},  when $e$ is unitary, $\gamma $ and $\theta_2$ take the following simple forms:
\begin{equation}
(\theta_2)_{ij} = \sum_{k=1}^n \overline{T^k_{ij}} \varphi_k, \ \ \ \ \gamma_{ij} = \sum_{k=1}^n ( T_{ik}^j \varphi_k - \overline{T^i_{jk}} \overline{\varphi}_k )
\end{equation}

Next, let us recall Gauduchon's {\em torsion $1$-form} $\eta$ which is
defined to be the trace of $\gamma'$ (\cite{Gauduchon}). Under any
frame $e$, it has the expression:
\begin{equation}
\eta = \mbox{tr}(\gamma') = \sum_{i,j=1}^n T^i_{ij}\varphi_j
\end{equation}
A direct computation shows that
\begin{equation}
\partial \omega^{n-1} = -2 \ \eta \wedge \omega^{n-1},
\label{formula 15}
\end{equation}
where $\omega $ is the K\"ahler (or metric) form of $g$. The metric $g$ is said to be {\em balanced} if
$\omega^{n-1}$ is closed. The above identity shows that $g$ is balanced
if and only if $\eta =0$. When $n=2$, $\eta =0$ means $\tau =0$, so
balanced complex surfaces are K\"ahler. But in dimension $n\geq 3$, $\eta $
contains less information than $\tau$.

\vsv

Under our notations, the components of the Chern and  Riemannian curvature tensors are
given by
\begin{equation}
R^c_{i\overline{j}k\overline{l}}=\sum_{p=1}^n \Theta_{ip}(e_k,
\overline{e}_l)g_{p\overline{j}}, \ \ \ R_{abcd}=\sum_{e=1}^{2n}
\hat{\Theta}_{ae}(e_c,e_d)g_{eb}  \label{formula 16}
\end{equation}
where $a, \ldots , e$ are between $1$ and $2n$, with
$e_{n+i}=\overline{e}_i$. Note that
$g_{ij}=g_{\overline{i}\overline{j}}=0$, so we have
\begin{eqnarray}
R_{i\overline{j}k\overline{l}} & = & \sum_{p=1}^n
(\Theta_1^{1,1})_{ip}(e_k, \overline{e}_l)g_{p\overline{j}}, \ \
R_{\overline{i}\overline{j}kl} \ = \  \sum_{p=1}^n
(\Theta_2^{2,0})_{ip}(e_k, e_l)  g_{p\overline{j}}
\label{formula 17}\\
R_{\overline{i}\overline{j} k\overline{l}} & = & R_{k\overline{l}
\overline{i} \overline{j} } \ = \ \sum_{p=1}^n (\Theta_2^{1,1})_{ip}
(e_k, \overline{e}_l) g_{p\overline{j}} \ =
\ \sum_{p=1}^n (\Theta_1^{0,2})_{kp}(\overline{e}_i, \overline{e}_j)g_{p\overline{l}}  \label{formula 18} \\
R_{\overline{i}\overline{j}\overline{k}\overline{l}} & = & R_{ijkl}
\ = \ 0  \label{formula 19}
\end{eqnarray}
The last line is because $\Theta_2^{0,2}=0$ by Lemma 1 of \cite{YZ}, a  property for general Hermitian metric discovered by Gray in
\cite{Gray} (Theorem 3.1 on page 603). The following lemma  is taken from \cite{YZ} (Lemma 7):

\vsv

\begin{lemma}  \label{lemma 1}
Let $(M^n,g)$ be a Hermitian manifold. Let $e$
be a unitary frame in $M$, then
\begin{eqnarray}
2T^k_{ij,\ \overline{l}} & = &
R^c_{i\overline{k}j\overline{l}} - R^c_{j\overline{k}i\overline{l}}
\label{formula 21}\\
R_{ijk\overline{l}} \ & = & T^l_{ij,k} + T^l_{ri} T^r_{jk} -
T^l_{rj} T^r_{ik}  \label{formula 22}
\end{eqnarray}
\begin{eqnarray}
R_{ij\overline{k}\overline{l}} & = &
T^l_{ij,\overline{k}} - T^k_{ij,\overline{l}} +
2T^r_{ij} \overline{T^r_{kl}} + T^k_{ri} \overline{ T^j_{rl} } +
T^l_{rj} \overline{ T^i_{rk} } - T^l_{ri} \overline{ T^j_{rk} } -
T^k_{rj} \overline{ T^i_{rl} }   \label{formula 23}\\
R_{i\overline{j}k\overline{l}} & = &
R^c_{i\overline{j}k\overline{l}} - T^j_{ik,\overline{l}} -
\overline{ T^i_{jl,\overline{k}} } + T^r_{ik} \overline{ T^r_{jl} }
- T^j_{rk} \overline{ T^i_{rl} } - T^l_{ri} \overline{ T^k_{rj} }
\label{formula 24}
\end{eqnarray}
where the index $r$ is summed over $1$ through $n$, and the index after the comma stands for covariant derivative with respect to the Chern connection $\nabla^c$.
\end{lemma}

Note that these formula were the main computational tools used in \cite{YZ}. However, in our present situations, we would prefer to use the Bismut connection $\nabla^b$ and Bismut covariant differentiation instead of $\nabla^c$. To give the precise formula, let us start with the description of $\nabla^b$ under our frame work. Again let us fix a Hermitian manifold $(M^n,g)$ with $n\geq 2$.

\vsv

 Recall that the Bismut connection $\nabla^b$ of $(M^n,g)$ is the unique connection that is compatible with the metric and the almost complex structure, and  its $(3,0)$ torsion is skew-symmetric. The existence and  uniqueness of $\nabla^b$ is proved by Bismut in \cite{Bismut}. Again let us fix a local type $(1,0)$ tangent frame $e$ and let $\varphi$ be its dual coframe. Since $\nabla^bJ=0$, we can write
$
\nabla^b e_i = \sum_{j=1}^n \theta^b_{ij} e_j
$.
We have the following:

\begin{lemma}  \label{lemma 2}
Under any frame $e$ of type $(1,0)$ tangent vectors, the components of the Bismut connection $\nabla^b$ are given by
\begin{equation}
\theta^b = \theta + 2 \gamma
\end{equation}
\end{lemma}

\begin{proof} Clearly, the connection $\nabla^b$ defined by $\theta + 2\gamma $ is compatible with the metric and the almost complex structure, so we just need to verify that its $(3,0)$ torsion is skew-symmetric. We have
$T^b(X,Y) = T^c(X,Y) + 2 \gamma_XY - 2\gamma_Y X$. Since $T^c(e_i , \overline{e}_j)=0$ and $T^c(e_i,e_j)= 2\sum_k T_{ij}^ke_k$, so  under any unitary frame $e$ by $(10)$ we get
\begin{equation}
 T^b(e_i, e_j) = - 2 \sum_k T_{ij}^k e_k, \ \ \ \ T^b(e_i, \overline{e}_j) = 2\sum_k ( T_{ik}^j \overline{e}_k - \overline{T^i_{jk}} e_k ),
 \end{equation}
and from this it is easy to verify that $\langle T^b(X,Y), Z\rangle = - \langle T^b(X,Z), Y\rangle$ for any tangent vectors $X$, $Y$, and $Z$. So by the uniqueness we know $\nabla^b$ must be the Bismut connection.
\end{proof}

Using Lemma 2, we can compute the curvature of the Bismut connection for a given Hermitian metric. Let us illustrate this by considering the following example:

\begin{lemma}  \label{lemma 3}
Consider the Hermitian metric on ${\mathbb C}^2\setminus \{ 0\}$ with K\"ahler form $\omega = \frac{\sqrt{-1}} {|z|^2}  \partial \overline{\partial }  |z|^2$, where $z=(z_1,z_2)$ is the standard coordinate of ${\mathbb C}^2$, and $|z|^2=|z_1|^2+|z_2|^2$. We claim that the curvature of its Bismut connection is everywhere zero.
\end{lemma}

\begin{proof}  Let $e$ be the unitary frame
$ e_i = |z| \frac{\partial }{\partial z_i} $ for $i=1$, $2$, its dual coframe is $\varphi_i= \frac{1}{|z|}dz_i$. Under the frame $e$, we have
$$ \theta = (\overline{\partial } - \partial )\log |z| I, \ \ \ \ \tau = - 2 \partial \log |z| \wedge \varphi .$$
So we get $T^1_{12} = \frac{\overline{z}_2} {2|z|}$ and  $T^2_{12} = - \frac{\overline{z}_1} {2|z|}$, thus
$$ \gamma = \frac{1}{2|z|^2} \left[ \begin{array}{ll} \overline{z}_2 dz_2 - z_2 d \overline{z}_2    &    z_2 d \overline{z}_1 -    \overline{z}_1 dz_2  \\   z_1 d \overline{z}_2-    \overline{z}_2 dz_1   &  \overline{z}_1 dz_1 - z_1 d \overline{z}_1 \end{array} \right] , \ \ \ \mbox{and} \ \ \
\ \ \theta^b = \theta + 2\gamma = \frac{1}{2|z|^2} \left[ \begin{array}{cc}   A & 2B \\ -2\overline{B} & - A  \end{array} \right]    , $$
where
$$ A =  z_1 d \overline{z}_1 + \overline{z}_2 dz_2  -  z_2 d \overline{z}_2- \overline{z}_1 dz_1 , \ \ \ \ B =  z_2 d \overline{z}_1 - \overline{z}_1 dz_2.$$
Now it is a straight forward computation to verify that $\Theta^b = d\theta^b - \theta^b \wedge \theta^b = 0$, and we omit it here. So $\omega$ is Bismut flat. \end{proof}

Recall that a compact complex surface $M^2$ is called a Hopf surface, if its universal cover is ${\mathbb C}^2\setminus \{0\}$. A Hopf surface $M^2$ is called a primary Hopf surface, if $\pi_1(M) \cong {\mathbb Z}$. Kodaira \cite{Kodaira} proved that all Hopf surfaces are finite undercovers of primary Hopf surfaces, and all primary Hopf surfaces are diffeomorphic to $S^3\times S^1$, in the form $M_{a,b}$ or $M_{a;m}$ below, where $a$ and $b$ are complex numbers satisfying $0<|a|\leq |b|<1$ and $m\geq 2$ is an integer. Here
$$ M_{a,b} = ({\mathbb C}^2 \setminus \{ 0\} )/{\mathbb Z}\phi , \  \ M_{a; m} = ({\mathbb C}^2 \setminus \{ 0\} )/{\mathbb Z}\psi ,
\ \ \ \mbox{where} \ \ \ $$
$$ \phi (z_1, z_2) = (az_1 , bz_2),  \ \ \ \mbox{and} \ \ \psi (z_1, z_2) = (az_1 , z_1^m+ a^mz_2) , $$
with $(z_1,z_2)$  the standard coordinate of ${\mathbb C}^2$.

\vsv

We will call a Hopf surface $M$ covered by $M_{a,b}$ with $|a|=|b|$ an {\em isosceles Hopf surface.} Its fundamental group is an extension of ${\mathbb Z}\phi $ by a finite group $F$, where $\phi (z_1, z_2)=(az_1,bz_2)$ with $0<|a|=|b|<1$, and $F$ is a finite subgroup of $U(2)$ (and in fact $F$ is a subgroup of $U(1)\times U(1)$ when $a\neq b$, see \cite{Kato} for more details). Clearly, any element of $\pi_1(M)$ preserves $\omega$ in Lemma 3, so the metric descends down to $M$, and we  get

\begin{lemma} \label{Hopf}
Any isosceles Hopf surface  admits a Bismut flat Hermitian metric.
\end{lemma}

Conversely, as a consequence of our main theorem, we shall see that
any compact Bismut flat Hermitian surface is either an isosceles
Hopf surface when it is non-K\"ahler, or, when it is K\"ahler,  a
flat complex torus or hyperelliptic surface.

\vsv

We conclude this section by stating the following well known result
with a sketched proof.

\begin{lemma}  \label{lemma 5}
Let $(M^n,g)$ be a Hermitian manifold whose Bismut connection
$\nabla^b$ is flat. Then given any $p\in M$, there exists a
neighborhood $p\in U\subseteq M$ and type $(1,0)$ unitary frame $e$
in $U$ which is $\nabla^b$-parallel, namely, $\nabla^b e_i=0$ in $U$
for each $1\leq i \leq n$.

\end{lemma}

\begin{proof} Since the curvature of the connection
$\nabla^b$ is everywhere zero, the entries of the connection matrix
forms a completely integrable system, therefore there will be local
frame of the real tangent bundle of $M$ which is
$\nabla^b$-parallel. But $\nabla^b J=0$, so we have type $(1,0)$
complex tangent frame that is $\nabla^b$-parallel.
\end{proof}

\vs

\vsv

\vsv

\vsv

\section{The Samelson spaces}

In this section, let us recall Samelson's construction \cite{S} of
left invariant complex structures on even-dimensional compact Lie
groups, which comes from a choice of a maximal torus, a complex
structure on the Lie algebra of the torus, and a choice of positive
roots for the Cartan decomposition. We will also discuss Pittie's
theorem \cite{P} which states that all left invariant complex
structures on such groups are actually obtained this way. In \S 5 of
\cite{AI}, Alexandrov and Ivanov gave a nice description of both
results. Here for the convenience of the readers, we include some of
their arguments briefly, on those statements that we will need in
our later discussions.

\vsv

First let us begin with the following result, which was observed by
Alexandrov and Ivanov (\cite{AI}, p.263), and might be known to
other experts as well.  We include a proof here for the
convenience of the readers.

\begin{lemma} {\bf (Alexandrov-Ivanov)} \label{lemma AI}
Let $G$ be an even dimensional connected Lie group equipped with a
bi-invariant metric $g=\langle \ , \rangle $ and a left invariant
complex structure $J$ which is compatible with $g$. Then the
Hermitian manifold $(G, J, g)$ is Bismut flat.
\end{lemma}

\begin{proof} Let $\{ e_1, \ldots ,
e_n\}$ be a unitary frame of left-invariant vector fields on $G$ of
type $(1,0)$. It suffices to show that $\nabla^b e_j=0$ for each
$j$. Since the metric is bi-invariant, it is well-known that the
Riemannian connection $\nabla$ is given by $\nabla_X Y = \frac{1}{2}
[X,Y]$ for left invariant vector fields. The integrability condition
on the complex structure means that we have $[e_i,e_j]=\sum_k
C_{ij}^k e_k$ for some constants $C^k_{ij}$, and since
\begin{eqnarray*}
\langle [\overline{e_i} , e_j ] , e_k \rangle & = & \langle [e_j, e_k ] , \overline{e_i}  \rangle \ = \ C_{jk}^i \\
\langle [\overline{e_i} , e_j ] , \overline{e_k} \rangle & = & - \langle [\overline{e_i}, \overline{e_k}], e_j  \rangle \ = \ - \overline{C_{ik}^j}
\end{eqnarray*}
we get $ [\overline{e_i} , e_j ] = \sum_k ( C_{jk}^i \overline{e_k} - \overline{C_{ik}^j} e_k)$. By Lemma 2, we have
\begin{eqnarray*}
\nabla^b_{e_i}e_j & = & \nabla^c_{e_i}e_j + 2\gamma_{e_i}e_j \ = \  \nabla_{e_i}e_j - \gamma_{e_i}e_j + 2 \gamma_{e_i}e_j \\
& = & \frac{1}{2} [ e_i, e_j] + \sum T_{ji}^k e_k \ = \ \sum (\frac{1}{2}C_{ij}^k - T_{ij}^k)e_k \\
\nabla^b_{\overline{e_i}}e_j & = & \nabla^c_{\overline{e_i}}e_j + 2\gamma_{\overline{e_i}}e_j \ = \   \nabla_{\overline{e_i}}e_j - \gamma_{\overline{e_i}}e_j - \sum \overline{ (\theta_2)_{jk}(e_i) } \overline{e_k} + 2 \gamma_{\overline{e_i}}e_j \\
& = & \frac{1}{2} [ \overline{e_i}, e_j] - \sum (\overline{T_{ki}^j} e_k + T_{jk}^i \overline{e_k} ) \ = \  \sum (\frac{1}{2} C_{jk}^i - T_{jk}^i )\overline{e_k} - \sum (\frac{1}{2} \overline{C_{ik}^j} - \overline{T_{ik}^j} )e_k
\end{eqnarray*}
Since $\nabla^b_{\overline{e_i}}e_j = \sum_k \theta^b_{jk}(\overline{e_i}) e_k$, we know from the last line above that $\frac{1}{2}C_{jk}^i = T_{jk}^i$, hence $\nabla^be_j=0$ for each $j$, and $\nabla^b$ is flat.
\end{proof}

Next, let $G$ be a connected Lie group equipped with a bi-invariant
metric $\langle \ , \rangle $. Denote by ${\mathfrak g}$ the Lie
algebra of $G$, and also denote by $\langle \ , \rangle $ the inner
product on ${\mathfrak g}$ induced by the metric of $G$. Since the
metric is bi-invariant, we have
$$ \langle [X,Y],Z\rangle = - \langle [X,Z],Y\rangle  $$
for any vectors $X$, $Y$, $Z$ in ${\mathfrak g}$. So if ${\mathfrak
a} \subset {\mathfrak g}$ is an ideal, denote by ${\mathfrak
a}^{\perp }$ its perpendicular compliment in ${\mathfrak g}$. By
letting $Y\in {\mathfrak a}$ and $Z\in {\mathfrak a}^{\perp }$ in
the above identity, we see that ${\mathfrak a}^{\perp }$ is also an
ideal in ${\mathfrak g}$. So we can write the Lie algebra as the
orthogonal direct sum of simple ideals. This leads to the following
Milnor's Lemma (see Lemma 7.5 of \cite{M}):

\begin{lemma}{\bf (Milnor) }  \label{Milnor}
Let $G$ be a simply-connected Lie group with a bi-invariant metric $\langle \ , \rangle$. Then $G$ is isomorphic and isometric to the product $G_1 \times \cdots \times G_r \times {\mathbb R}^k$ where each $G_i$ is a simply-connected compact simple Lie group and ${\mathbb R}^k$ is the additive vector group with the flat metric. Here  $0\leq k\leq \dim(G)$.
\end{lemma}

As is well known, the simply-connected compact simple Lie groups are fully classified, they are:

$A_n = SU(n+1)$, $n\geq 1$,  $\dim(A_n) = n(n+2)$;

$B_n = \mbox{Spin} (2n+1)$, $n\geq 2$, $\dim(B_n) = n(2n+1)$ ;

$C_n= \mbox{Sp}(2n)$, $n\geq 3$, $\dim(C_n) = n(2n+1)$;

$D_n = \mbox{Spin} (2n)$, $n\geq 4$, $\dim(D_n) = n(2n-1)$;

$E_6$, $\dim(E_6) = 78$;

$E_7$, $\dim(E_7) = 133$;

$E_8$, $\dim(E_8) = 248$;

$F_4$, $\dim(F_4) = 52$;

$G_2$, $\dim(G_2) = 14$.

\vsv

Note that the only ones in dimension less than $14$ are $M^3=SU(2)$, $M^8=SU(3)$, and $M^{10}=Spin(5)$. So only those three could  appear in a compact Bismut flat manifold of complex dimension less than or equal to $6$.

\vsv

Since any bi-invariant metric on a compact simple Lie group is a constant multiple of the Killing form, the bi-invariant metric on $G$ is unique up to constant multiples on the compact factors, and each $G_i$ is an Einstein manifold with positive Ricci curvature.

\vsv

Let us fix a simply-connected Lie group $G$ with a bi-invariant metric $\langle \ , \rangle $. We have $G=G_1 \times \cdots \times G_r\times {\mathbb R}^k$ as above. Let $G'=G_1 \times \cdots \times G_r\times T^k$ where $T_k$ is the torus. Then $G$ is the covering group of $G'$, and they share the same Lie algebra $\mathfrak g$.

\vsv

Now we assume that $\dim(G)$ is even. Note that the left invariant complex structures on $G$ or $G'$ are both in one-one correspondence with left invariant complex structures on $\mathfrak g$, which are linear maps $J: {\mathfrak g}  \rightarrow {\mathfrak g}$ such that $J^2=-I$ and
\begin{equation}
J ( [X,Y] -[JX,JY] ) = [JX,Y] + [X, JY]
\end{equation}
for any $X$, $Y$ in ${\mathfrak g}$.

\vsv

Samelson constructed left invariant complex structures $J$ on $G$
that is compatible with the metric, by choosing a maximal torus $K$
in $G'$, a complex structure on the Lie algebra $\mathfrak k$ of
$K$, and a choice of positive roots for the Cartan decomposition of
$\mathfrak g$. In \cite{AI}, \S 5, Alexandrov and Ivanov give a nice
description of Samelson's construction. Here we include a brief
account of it for the convenience of the  readers.

\vsv

Denote by  ${\mathfrak g}^c $ the complexification of ${\mathfrak
g}$. The existence of a compatible left invariant complex structure
$J$ is equivalent to the existence of a complex subspace ${\mathfrak
s} \subset {\mathfrak g}^c$, such that $\langle {\mathfrak s},
{\mathfrak s} \rangle =0$, ${\mathfrak s} \cap {\mathfrak g}=0$, and
${\mathfrak s} \oplus \overline{{\mathfrak s}} = {\mathfrak g}^c$.
Such a subspace is called a {\em Samelson subalgebra} of ${\mathfrak
g}^c$.

\vsv

Now let $K$ be a maximal torus of $G'$, and $\mathfrak k$ its Lie algebra. Denote by ${\mathfrak k}^c $ the complexification of ${\mathfrak k}$. When a set of positive roots $\alpha_1 , \ldots , \alpha_m$ is chosen, then it is well-known that one has the $\mbox{ad}(K)$-invariant decomposition
\begin{equation} {\mathfrak g}^c = {\mathfrak k}^c \oplus \sum_{j=1}^m {\mathfrak g}_{\alpha_j} \oplus \sum_{j=1}^m {\mathfrak g}_{-\alpha_j} ,
\end{equation}
where
\begin{equation} {\mathfrak g}_{\pm \alpha_j} = \{ Y \in {\mathfrak g}^c \ \mid \ [X,Y] = \pm 2\pi \sqrt{-1} \alpha_j(X)Y, \ \ \forall X \in {\mathfrak k} \}
\end{equation}
are the root spaces.

\vsv

Since $\dim(G)$ is even, we know that the abelian Lie algebra ${\mathfrak k}$ is even dimensional. So we can choose an almost complex structure on ${\mathfrak k}$ that is compatible with the metric. This means, we have a complex subspace ${\mathfrak a} \subset {\mathfrak k}^c$ such that $\langle {\mathfrak a}, {\mathfrak a} \rangle =0$, ${\mathfrak a} \cap {\mathfrak k}=0$, and ${\mathfrak a} \oplus \overline{{\mathfrak a}} = {\mathfrak k}^c$.

\vsv

Now one could simply take
\begin{equation}
{\mathfrak s} = {\mathfrak a} \oplus \sum_{j=1}^m {\mathfrak g}_{\alpha_j}
\end{equation}
to be the Samelson subalgebra. So on any even-dimensional Lie group
$G$ equipped with a bi-invariant metric, there always exists
compatible left invariant complex structures on $G$, constructed by
an arbitrary choice of an almost complex structure (compatible with
the metric) on the Cartan subalgebra plus the choice of a set of
positive roots in the root decomposition.

\vsv

Conversely, Pittie \cite{P} proved that, any left invariant complex
structure on $G$ is obtained this way. Again a nice description of
this is given  by Alexandrov and Ivanov in \S 5 of \cite{AI}, and we
also include a brief account of their argument here for readers'
convenience.

\vsv

Let ${\mathfrak s}$ be a Samelson subalgebra of ${\mathfrak g}^c$
corresponding to a left invariant complex structure $J$ on the
compact Lie group $G'$. Let
$$ {\mathfrak k } = \{ X \in {\mathfrak
g } \ \mid \  \mbox{ad}(X) ({\mathfrak s } ) \subseteq {\mathfrak s
} \} $$ be the set of all elements in ${\mathfrak g }$ that
preserves the decomposition ${\mathfrak g }^c = {\mathfrak s }
\oplus \overline{ {\mathfrak s }} $.  Then it is easy to see that
${\mathfrak k}$ is a $J$-invariant subalgebra of ${\mathfrak g}$,
hence ${\mathfrak k}$ is a complex Lie algebra. Let $K$ be a closed
connected Lie subgroup of $G'$ corresponding to ${\mathfrak k}$.
Then $K$ is a compact complex Lie group, thus a torus and is
abelian. We have an $\mbox{ad}(K)$-invariant orthogonal
decomposition $(26)$ with $\mathfrak{g}_{\alpha_j}$ given by $(25)$
and
$$ {\mathfrak a} = \{ Y \in {\mathfrak s} \  \mid \  [X,Y] =0 \ \ \forall X \in {\mathfrak k} \} .$$
It follows from the definition of ${\mathfrak k}$ that ${\mathfrak
k}^c = {\mathfrak a} \oplus \overline{{\mathfrak a}}$, so
${\mathfrak k}$ is a maximal abelian subalgebra of ${\mathfrak g}$
and $K$ is a maximal torus. So $(24)$ is satisfied, and $\pm
\alpha_1, \ldots , \pm \alpha_m$ are all the roots. Since
$[{\mathfrak s} , {\mathfrak s} ] \subseteq {\mathfrak s}$,  we know
that if $\alpha_i+\alpha_j$ is a root, then $\alpha_i+\alpha_j =
\alpha_l$ for some $l$. So we can take $\{ \alpha_1, \ldots ,
\alpha_m\}$ to be our set of positive roots. This shows that any
left invariant complex structure on a compact Lie group is
determined by a choice of a maximal torus, a choice of a complex
structure on the Lie algebra of the maximal torus, and a choice of
positive roots.

\vsv

As an application of the above characterization, let us consider
left invariant complex structures $J$ on  a simply-connected Lie
group $G$ equipped with a bi-invariant metric, in the special cases
when $\dim (G)=2n$ is small. Note that when $G={\mathbb R}^{2n}$,
the left invariant complex structures are just those identifications
of ${\mathbb R}^{2n} \cong {\mathbb C}^n$. In this case the metric
is K\"ahler, and vice versa. So for our discussion below let us
assume that $G$ is not the vector group.

\vsv

First let us start with $n=2$. In this case, $G$ has only one
choice:  $SU(2)\times {\mathbb R}$. Denote by $W$ a unit vector in
the factor ${\mathbb R}$. Note that in the Lie algebra ${\mathfrak
su}(2)$, the bracket is given by (twice of) the usual cross product,
namely, if $X$, $Y$, $Z$ forms a (positively oriented) orthonormal
basis of it, then $$ [X,Y]=2Z, \ \ [Y,Z]=2X, \ \ [Z,X]=2Y.$$
Therefore, a compatible left invariant complex structure $J$ on $G$
is determined by the choice of a unit vector $X$ in ${\mathfrak
su}(2)$, as the image of $W$ under $J$, and we must have $JY=Z$ if
$\{ X, Y, Z\}$ forms a positive orthonormal basis. Since ${\mathfrak
su}(2)\cong {\mathfrak so}(3)$, we know that such $J$ are all
isomorphic to each other. In other words, when $n=2$, the universal
cover of compact, non-K\"ahler Bismut flat surfaces is unique (up to
the change on the metric by a constant multiple): they are all
holomorphically isometric to ${\mathbb C}^2\setminus \{ 0\}$ equipped
with the metric $c \frac{\sqrt{-1}}{|z|^2} \partial
\overline{\partial } |z|^2$, where $c$ is a positive constant.

\vsv

Now let us look at the $n=3$ case. $G$ is either $SU(2)\times
{\mathbb R}^3$ or $SU(2)\times SU(2)$. Let $J$ be a compatible left
invariant complex structure on $G$. By the results of Samelson and
Pittie, we know that $JV\cap V\neq 0$ for the $V={\mathfrak su}(2)$
factor in $\mathfrak g$. So in the case when ${\mathfrak g} =
{\mathfrak su}(2) \oplus {\mathbb R}^3$, the complex structure must
be in the form: $JY=Z$, $JX=W_1$, and $JW_2=W_3$, where $\{ X, Y,
Z\}$ is an orthonormal basis of ${\mathfrak su}(2)$ and $\{ W_1,
W_2, W_3\}$ is an orthonormal basis of ${\mathbb R}^3$. This means
that $G$ is holomorphically isometric to $({\mathbb C}^2\setminus \{
0\} ) \times {\mathbb C}$.

\vsv

Similarly, when $ {\mathfrak g} = {\mathfrak su}(2) \oplus
{\mathfrak su}(2)$, we have orthonormal basis $\{ X, Y, Z\}$ for the
first factor and $\{ X_1, Y_1, Z_1\}$ for the second factor, such
that $JY=Z$, $JY_1=Z_1$, and $JX=X_1$. This is one particular complex
structure in the family of complex structures on $S^3\times S^3$
given by Calabi-Eckmann \cite{CE}, and for lack of better
terminologies, we will call it a {\em central Calabi-Eckmann threefold}
(see \S 1). Note that as Hermitian manifolds, such spaces are unique up
to the choice of two positive constants $c$, $c'$, so the metric on
the manifold is $g=(cg_0)\times (c'g_0)$, where $g_0$ is the
standard metric on $SU(2)=S^3$ with constant sectional curvature
$1$.

\vsv

For $n=4$, we have $G=SU(2)\times {\mathbb R}^5$, or $SU(2)\times
SU(2) \times {\mathbb R}^2$, or $SU(3)$. In the first case, since
$J$ has to have a non-trivial invariant part in the ${\mathfrak
su}(2)$ factor, there is only one direction in the Euclidean factor
that is $J$-involved with ${\mathfrak su}(2)$ , so $G$ is
holomorphically isometric to the product of ${\mathbb C}^2\setminus
\{ 0 \}$ with ${\mathbb C}^2$. In the case $SU(2)\times SU(2) \times
{\mathbb R}^2$, a maximal abelian subalgebra ${\mathfrak k}$ of
${\mathfrak g}$ would consist of one direction from each ${\mathfrak
su}(2)$ plus the two dimensional Euclidean factor. The choice of $J$
on ${\mathfrak k}$ may or may not respect the original splitting,
e.g., $J$ could be chosen to be
$$ JX = aY + b Z, \ \ JY = -a X - b W, \ \ JZ = -bX +aW, \ \ JW = bY - aZ, $$
where $\{ Z,W\}$ is an orthonormal basis of ${\mathbb R}^2$ and $X$,
$Y$ are unit vectors from the two ${\mathfrak su}(2)$ factors, perpendicular to the $J$-invariant part, and
$a$, $b$ are real constants satisfying $a^2+b^2=1$. Note that for
such a $J$ (when $ab\neq 0$), $G$ is not holomorphically isometric
to either the product of two copies of $ {\mathbb C}^2\setminus \{ 0
\}$, or the product of a central Calabi-Eckmann threefold and
${\mathbb C}$.

\vsv

One could apply similar analysis on $G$ in other small dimensions.
In our opinion, Samelson spaces provide an interesting class of
complex manifolds, whose differential geometric aspects could be
further studies and exploited.

\vs

\vsv

\vsv

\vsv

\section{The Bismut flat metrics}

In this section, let us assume that $(M^n,g)$ is a Hermitian manifold whose Bismut connection $\nabla^b$ is flat. We are only interested in the case when $g$ is not K\"ahler.

\vsv

By Lemma 5, locally there will always be $\nabla^b$-parallel frames. Such frames are obviously unique up to changes by constant matrices. Let us fix a $\nabla^b$-parallel, unitary local frame $e$, and denote by $\varphi$ its dual coframe. By Lemma 2, we have $\theta = - 2\gamma$, and the structure equations and the first Bianchi identity for $\nabla^c$ and $\nabla$ specialize into the following
\begin{lemma} \label{lemma 8}
On a Bismut flat Hermitian manifold $(M^n,g)$, under a local unitary $\nabla^b$-parallel frame $e$, it holds that \begin{eqnarray}
\partial \varphi & = & - \tau \ \  = \ \ ^t\!\gamma ' \wedge \varphi \\
\overline{\partial } \varphi & = & - 2 \ \overline{\gamma '} \wedge \varphi \\
\partial \gamma ' & = & -2 \gamma ' \wedge \gamma ' \\
0 & = &  \ ^t\! \gamma ' \wedge \ ^t\! \gamma ' \wedge \varphi  \\
0 & = & \overline{\partial } \ ^t\! \gamma ' \wedge \varphi - 2 \ \partial \overline{\gamma '} \wedge \varphi + 2\  \overline{\gamma '} \wedge  \ ^t\! \gamma ' \wedge \varphi + 2 \ ^t\! \gamma ' \wedge  \overline{\gamma '} \wedge  \varphi
\end{eqnarray}
\end{lemma}
\begin{proof}
The first two identities are immediate from the structure equations and the fact $\theta = -2\gamma$ since $e$ is $\nabla^b$-parallel. The third one is due to the fact that the $(2,0)$ part of $\Theta$ is zero, and the last two are direct consequence of the first Bianchi identity under the circumstance.
\end{proof}
Using the expression $\gamma '_{ij}= \sum_k T_{ik}^j \varphi_k$, we can rewrite the last three identities of Lemma 8 in terms of the torsion components $T_{ij}^k$ and their covariant derivatives with respect to $\nabla^b$:
\begin{eqnarray*}
T_{ik,l}^j - T_{il,k}^j & = &  2 \sum_{r} (  \ T_{ik}^r T_{rl}^j + T_{li}^r T_{rk}^j   +  T_{kl}^r T_{ri}^j ) \\
0 & = & \sum_r (T_{ij}^r T_{rk}^l  + T_{jk}^r T_{ri}^l + T_{ki}^r T_{rj}^l )  \\
T^i_{kl,\overline{j}} + \overline{ T^k_{ij,\overline{l}} } - \overline{ T^l_{ij,\overline{k}} } & = & 2 \sum_r ( T_{lr}^i \overline{T_{jr}^k} - T_{kr}^i \overline{T_{jr}^l} - T_{lr}^j \overline{T_{ir}^k} + T_{kr}^j \overline{T_{ir}^l} - T_{kl}^r \overline{T_{ij}^r} )
\end{eqnarray*}
for any $1\leq i,j,k,l \leq n$. Note that when $i$, $j$, $k$ are not all distinct, the right hand side of the middle equality is automatically zero, so this line holds true even when $n=2$.  From the first two, we know that $T_{ik,l}^j = T_{il,k}^j$, which implies $T_{ik,l}^j=0$ for all indices, since any trilinear form which is skew-symmetric with respect to its first two positions while symmetric with respect to its last two positions must be zero, as illustrated by
 $$ C_{ij,k} = - C_{ji,k} = - C_{jk,i} = C_{kj,i} = C_{ki,j} = - C_{ik,j} = - C_{ij,k}. $$
 For the last identity, let us denote the right hand side of the equality by $A_{kl}^{ij}$. It is skew-symmetric in $ij$, namely, $A^{ij}_{kl}+A^{ji}_{kl}=0$. So the identity implies that $ T^i_{kl,\overline{j}} = - T^j_{kl,\overline{i}}$, thus its left hand side is equal to $T^i_{kl,\overline{j}} + 2 \overline{T^k_{ij,\overline{l}}}$. Also, since $\overline{A^{kl}_{ij}} = A^{ij}_{kl}$, we know that $T^i_{kl,\overline{j}}  = \overline{T^k_{ij,\overline{l}}} = \frac{1}{3} A^{ij}_{kl}$. In summary, we have the following

\begin{lemma} \label{lemma 9}
On a Bismut flat Hermitian manifold $(M^n,g)$, under a local unitary $\nabla^b$-parallel frame $e$, it holds
\begin{eqnarray}
0 & = &  T_{ik,l}^j   \\
0 & = & \sum_r (T_{ij}^r T_{rk}^l  + T_{jk}^r T_{ri}^l + T_{ki}^r T_{rj}^l )
\end{eqnarray}
\begin{eqnarray}
T^i_{kl,\overline{j}} & = & - T^j_{kl,\overline{i}} \ = \ \overline{T^k_{ij,\overline{l}}} \\
T^i_{kl,\overline{j}} & = & \frac{2}{3} \sum_r ( T_{lr}^i \overline{T_{jr}^k} - T_{kr}^i \overline{T_{jr}^l} - T_{lr}^j \overline{T_{ir}^k} + T_{kr}^j \overline{T_{ir}^l} - T_{kl}^r \overline{T_{ij}^r} ) \\
 \sum_r \eta_{r,\overline{r}} & = & \frac{2}{3} (|T|^2 - 2|\eta |^2)
\end{eqnarray}
for any $1\leq i,j,k,l \leq n$, where $r$ is summed from $1$ to $n$, and the index after the comma means covariant derivative with respect to $\nabla^b$.
\end{lemma}

Note that the last identity is obtained by letting $i=k$, $j=l$, and sum up in $(35)$.

\vsv

Write $\eta = \sum_i \eta_i \varphi_i$. By $(28)$, we have
$ \overline{\partial } \eta = - \sum_{i,j=1}^n (\eta_{i,\overline{j}} + 2 \sum_p \eta_p \overline{T^i_{jp}}) \varphi_i\wedge \overline{\varphi_j}$, so
$$  \sqrt{-1} \ \overline{\partial } \eta \wedge \omega^{n-1} = - \sum_i (\eta_{i,\overline{i}} + 2|\eta_i|^2) \frac{\omega^n}{n},$$
where $\omega$ is the K\"ahler form of the metric of $M^n$. On the other hand, by $(12)$, we have
$$ \partial \overline{\partial }\omega^{n-1} = 2 (\overline{\partial } \eta + 2 \eta \wedge \overline{\eta })\wedge \omega^{n-1},$$ thus by $(36)$ we get the following:

\begin{lemma} \label{lemma10}
On a Bismut flat manifold  $(M^n,g)$, it holds
\begin{equation}
- \sqrt{-1} \  \partial \overline{\partial }\omega^{n-1} = \frac{2}{n} (\sum_i \eta_{i,\overline{i}}) \ \omega^n = \frac{4}{3n} ( |T|^2 - 2|\eta |^2 ) \ \omega^n
\end{equation}
\end{lemma}

From this identity, we immediately get that, if the Bismut flat
manifold $M$ is balanced, then $T=0$, i.e., it is K\"ahler. Also,
when $M$ is compact, the integral of the right hand side of the
above equation is zero. Note that under the frame $\{ e, \overline{e}\}$, the torsion tensor $T^c$ of the Chern connection takes the form
$$ T^c(e_i, e_j) = 2\sum_k T^k_{ij} e_k, \ \ \ \ T^c(e_i, \overline{e_j})= 0, \ \ \ \ T^c(\overline{e_i}, \overline{e_j})=2\sum_k \overline{T^k_{ij}} \overline{e_k}, $$
so $|T^c|^2= 8\sum_{i,j,k} |T^k_{ij}|^2 = 8|T|^2$, thus Theorem 2 is proved.

\vsv

Note that when $n=2$, the torsion tensor has only two components:
$T^1_{12}$ and $T^2_{12}$. The Gauduchon $1$-form has coefficients
$\eta_1 = -T^2_{12}$ and $\eta_2 = T^1_{12}$, and we always have
$|T|^2=2|\eta|^2$ when $n=2$. So $\eta_{1, \overline{1}} + \eta_{2,
\overline{2}} =0$ by $(36)$. On the other hand, by $(34)$, $\eta_{1,
\overline{1}} = - T^2_{12, \overline{1}} = T^1_{12, \overline{2}} =
\eta_{2, \overline{2}}$, so both are zero, and we get $T^i_{jk,
\overline{l}}=0$ for all indices. Hence both $T^1_{12}$ and
$T^2_{12}$ are constants. This leads to a proof of Theorem 5 in the $n=2$ case if we
follow the proof of Theorem 1 in the next page.

\vsv

When a Bismut flat manifold $(M^n,g)$ is compact, however, we will show that all the $T^i_{jk}$ (under a local Bismut parallel unitary frame) are indeed constants. The reason is due to the following simple observation that, the globally defined function $|T|^2=\sum_{i,j,k} |T^i_{jk}|^2$ on $M$ is plurisubharmonic. Note that the sum is independent of the choice of the local unitary frames, so the function is globally defined.

\begin{lemma}
On a Bismut flat manifold $(M^n,g)$, the square norm of the torsion tensor (for the Chern connection) is plurisubharmonic, and under a local unitary Bismut parallel frame $e$, it holds that
\begin{equation}
\partial \overline{\partial } |T|^2 = \sum_{i,j,k,l, m} T^i_{jk, \overline{l}} \ \overline{ T^i_{jk, \overline{m}} } \ \varphi_m \wedge \overline{\varphi_l}
\end{equation}
In particular, if $M$ is compact, then all $T^i_{jk}$ are constants.
\end{lemma}

\begin{proof}
Let $e$ be a local tangent frame of type $(1,0)$ vector fields, that is unitary and $\nabla^b$-parallel. Let $\varphi$ be the coframe of $(1,0)$ forms dual to $e$. Denote by $T^i_{jk}$ the components under the frame $e$ of the torsion tensor of the Chern connection. From the proof of Lemma 6, we have
$$ [ e_m, {\overline{e_l}} ] = 2 \sum_p ( \overline{T^m_{lp}} e_p - T^l_{mp} \overline{e_p} ) .$$
So by $(32)$, we get
$$ T^i_{jk,\overline{l}m} = [e_m, \overline{e_l} ] T^i_{jk} = -2 \sum_p T^l_{mp} T^i_{jk,\overline{p}}. $$
Also, by $(28)$ in Lemma 8, we know that
$$ \partial \overline{\varphi_p} = - 2 \sum_l \gamma'_{pl} \overline{\varphi_l} = -2 \sum_{m,l} T^l_{pm} \varphi_m \wedge \overline{\varphi_l}.$$
So by $(32)$,  we have
\begin{eqnarray*}
\partial \overline{\partial } |T|^2  & = & \partial \sum T^i_{jk,\overline{l}} \overline{T^i_{jk}} \overline{\varphi_l} \\
& = & \sum T^i_{jk, \overline{l}} \overline{T^i_{jk, \overline{m}} } \ \varphi_m \wedge \overline{\varphi_l} + \sum T^i_{jk, \overline{l}m} \overline{T^i_{jk}} \ \varphi_m \wedge \overline{\varphi_l} + \sum T^i_{jk,\overline{p}} \overline{T^i_{jk}} \ \partial \overline{\varphi_p} \\
& = & \sum T^i_{jk, \overline{l}} \overline{T^i_{jk, \overline{m}} } \ \varphi_m \wedge \overline{\varphi_l} - 2 \sum ( T^l_{mp} + T^l_{pm}) T^i_{jk,\overline{p}} \overline{T^i_{jk}} \ \varphi_m \wedge \overline{\varphi_l} \\
& = & \sum T^i_{jk, \overline{l}} \overline{T^i_{jk, \overline{m}} } \ \varphi_m \wedge \overline{\varphi_l} \ \geq \ 0
\end{eqnarray*}
When $M$ is compact, using any Gauduchon metric $\tilde{\omega }$ on $M$, we know that the function $|T|^2$ has to be a constant, so $T^i_{jk, \overline{l}} =0$ for all indices, thus all $T^i_{jk}$ are constants.
\end{proof}

Now we are ready to prove Theorem 1.

\vs

\begin{proof} [\textbf{Proof of Theorem 1.}]
Let $(M^n,g)$ be a compact Bismut flat manifold. Given any $p\in M$, let $e$ be a unitary $\nabla^b$-parallel frame of $(1,0)$ tangent vectors in a neighborhood of $p$, with $\varphi$ the dual coframe. By Lemma 11, all the components $T^i_{jk}$ of the torsion tensor under $e$ are constants. Since $\nabla^b e_i=0$, we get from  $(22)$ in the proof of Lemma 2 the following
\begin{eqnarray*}
 [ e_i, e_j ]  & = &  - T^b(e_i,e_j) \ = \  2 \sum   T^k_{ij}  e_k \\
 \mbox{}  [ e_i , \overline{e_j} ]   &  =  &    - T^b ( e_i,  \overline{e_j} ) \ =  \ 2 \sum ( \overline{ T^i_{jk} } e_k -  T^j_{ik}  \overline{e_k} )
\end{eqnarray*}
It is easy to verify that
\begin{equation}
 \langle [X,Y], Z\rangle = - \langle [X,Z],Y\rangle
 \end{equation}
hold for any $X$, $Y$, $Z$ in $\{ e_1, \ldots , e_n, \overline{e}_1,
\ldots , \overline{e}_n\} $. If we write $\varphi_i =
\frac{1}{\sqrt{2}} (\phi_i + \sqrt{-1} \phi_{n+i})$, then it is
straight forward to check that $\{ \phi_i \}_{i=1}^{2n}$ form the
left invariant forms for a local Lie group, with left invariant
metric and complex structure, and by $(39)$ we see that the metric
is actually bi-invariant.

\vsv

So lifting the metric and complex structure to the universal
covering space $\widetilde{M}$ of $M$, we know that $\widetilde{M}$
is a connected, simply-connected Lie group of even (real) dimension,
equipped with a bi-invariant metric, and a compatible left invariant
complex structure. In other words, $\widetilde{M}$ is a Samelson
space.

\vsv

Let us denote by $\Gamma$ the deck transformation group. By Milnor's
Lemma, we know that $\widetilde{M}$ is isomorphic and isometric to
the product $G\times {\mathbb R}^k$, where $G$ is a simply-connected
compact semisimple Lie group, equipped with a bi-invariant metric,
and ${\mathbb R}^k$ is the vector group, with the flat Euclidean
metric. Note that for each simple factor of $G$, the bi-invariant
forms are all proportional to the Killing form, so as a Riemannian
manifold it is Einstein with positive Ricci curvature. So the
${\mathbb R}^k$  corresponds to the kernel foliation of the
Riemannian curvature tensor, the so-called nullity foliation.

\vsv

Since the elements of $\Gamma $ are isometries, they preserve the
nullity foliation and its perpendicular compliment, therefore we
know that each $\gamma$ in $\Gamma$ must be in the form $\gamma
(x,y) = (\gamma_1(x), \gamma_2(y))$ for any $(x,y)\in G\times
{\mathbb R}^k$, with $\gamma_1\in I(G)$ and $\gamma_2 \in I({\mathbb
R}^k)$ in the isometry group of the factors.

\vsv

For $i=1$, $2$, let us denote by $\pi_i: \Gamma \rightarrow
\Gamma_i$ the projection maps, with $\Gamma_i$ the image group.

\vsv

Denote by $A$ the kernel of $\pi_2: \Gamma \rightarrow \Gamma_2$.
Since $\Gamma $ has discrete orbit, and $G$ is compact, we know that
$A$ must be a finite subgroup of $G$. For any $\gamma \in \Gamma$
and any $a\in A$, we have $\pi_2(\gamma a \gamma^{-1}) =1$, so the
map $\iota_{\gamma }(a)=\gamma a\gamma^{-1}$ is an automorphism of
$A$, and we get a group homomorphism $\iota: \Gamma \rightarrow
Aut(A)$. Since $Aut(A)$ is finite, we may replace $\Gamma$ by the
kernel of $\iota$, a normal subgroup of finite index, which amounts
to replacing $M$ by a finite unbranched cover of it, in this way we
may assume that $\iota $ is trivial, that is, $A$ in contained in
the center of $\Gamma$.

\vsv

Now since both $M$ and $G$ are compact, it is easy to see that
$\Gamma_2$ acts discretely and co-compactly on ${\mathbb R}^k$. So
by Bieberbach Theorem, there exists a normal subgroup $\Gamma_2'
\subseteq \Gamma_2$ of finite index,  such that $\Gamma_2'\cong
{\mathbb Z}^k$ is a lattice. If we  replace $\Gamma$ by
$\pi_2^{-1}(\Gamma_2')$, which amounts to replacing $M$ by another
finite unbranched cover of it, we may assume that $\Gamma_2\cong
{\mathbb Z}^k$ is a lattice in ${\mathbb R}^k$. In particular,
$\Gamma_2$ is abelian. So now we have the exact sequence
$$ 1\rightarrow A \rightarrow \Gamma \rightarrow \Gamma_2 \rightarrow 1,$$
where $\Gamma_2\cong {\mathbb Z}^k$ and $A$ is a finite group
contained in the center of $\Gamma$.

\vsv

Note that the commutator group $[\Gamma , \Gamma ]$ is contained in $A$ since $\Gamma_2$ is abelian. For any $b$, $c$ in $\Gamma$, we have $bcb^{-1}=ac$ for some elements $a\in A$. From this, we know that for any positive integer $n$, $b c^n b^{-1} = (ac)^n = a^n c^n$, so
$$ b^n c^n b^{-n} = b^{n-1} (bc^nb^{-1})b^{-(n-1)} = a^n b^{n-1}c^n b^{-(n-1)} = \cdots = a^{n^2} c^n$$
Therefore, $[b^n, c^n ] = [b,c]^{n^2}$.

\vsv

Now let $\{ \gamma_1 , \ldots , \gamma_k\}$ be a subset in $\Gamma$, such that $\{ t_1, \ldots , t_k\}$ is a set of generators in $\Gamma_2 \cong {\mathbb Z}^k$, where $t_i=\pi_2(\gamma_i)$. Let $n$ be a positive integer that is a multiple of the order of $A$. Let $\Gamma_2'' \subseteq \Gamma_2$ be generated by $\{ nt_1, \ldots , nt_k\}$, and let $\Gamma'' = \pi_2^{-1} (\Gamma_2'') $. Then $\Gamma''$ is generated by the set $A\cup \{ \gamma_1^n, \ldots , \gamma_k^n\}$.
The commutators of any two elements of this union set is trivial by the above identity. So $\Gamma''$ is abelian, and there is a homomorphism from it onto its torsion part.

\vsv

In summary, we can replace the original deck transformation group $\Gamma$ by a finite sequence of successive normal subgroup of finite index,  so in the end we may assume that the map $\pi_2: \Gamma \rightarrow \Gamma_2$ is injective, and $\Gamma_2\cong {\mathbb Z}^k$ is a lattice in ${\mathbb R}^k$. By letting $\rho = \pi_1\circ \pi_2^{-1}$, we get a homomorphism from ${\mathbb Z}^k$ into $\Gamma_1 \subseteq I(G)$ such that the elements of $\Gamma \cong{\mathbb Z}^k$ take the form
$$ \gamma_t (x,y) = (\rho(t)(x), y+t), \ \  \ \ \forall \  (x,y) \in G\times {\mathbb R}^k ,$$
where $t\in {\mathbb Z}^k$. We will denote this group by $\Gamma_{\rho}$  and write $M_{\rho } = (G\times {\mathbb R}^k)/\Gamma_{\rho }$. $M_{\rho }$ is a finite unbranched cover of the original $M$ that we started with.

\vsv

To see that $M_{\rho}$ is diffeomorphic to $G\times T^k$, where $T^k={\mathbb R}^k / {\mathbb Z}^k$ is the torus, let us start from the isometry group $I(G)$ of $G$. Since $G$ is compact, $I(G)$ is a compact Lie group. Let $\{ v_1, \ldots , v_k\}$ be a set of generators of ${\mathbb Z}^k$. Then any $y\in {\mathbb R}^k$ can be uniquely written as $y=t_1v_1+ \cdots + t_kv_k$ where $t_1, \ldots , t_k$ are real numbers.  For each $\rho (v_i)$ in $I(G)$, let $\psi_t^i$, $t\in {\mathbb R}$,  be a $1$-parameter subgroup of $I(G)$, such that $\psi^i_1=\rho (v_i)$.

\vsv

Define a diffeomorphism $\Psi$ from $ G\times {\mathbb R}^k$ onto itself by letting $\Psi (x,y) = (\psi^1_{t_1} \circ \cdots \circ \psi^k_{t_k} (x), y)$,  where $y=t_1v_1+ \cdots + t_kv_k$. Then $\Psi (x, y+v_i) = \gamma_i\circ \Psi (x,y)$, where $\gamma_i = (\rho (v_i), v_i) \in \Gamma_{\rho}$. So $\Psi$ descends down to a diffeomorphism from $G\times T^k$ onto the manifold $M_{\rho}$.  This completes the proof of Theorem 1.
\end{proof}

Note that when the image of $\rho$ is finite, then we can use its kernel to be the new deck transformation group, thus reducing to the $\rho =0$ case. In this case a finite cover of $M$ becomes the compact Lie group $G \times T^k$. When the image of $\rho$ is infinite, since it is abelian, we can ignore the torsion part (again by lifting to a finite cover) and assume that ${\mathbb Z}^k$ is the direct sum of two free abelian groups, with one summand being the kernel of $\rho$, and with $\rho$ being injective on the other summand.

\vsv

For a compact Bismut flat manifold $M^n$, since the local unitary $\nabla^b$-parallel frames are unique up to changes by constant unitary matrices, we get the Bismut holonomy map which is a homomorphisms $h: \pi_1(M) \rightarrow U(n)$. When the image group of $h$ is finite, then a finite unbranched cover $M'$ of $M^n$ has a global unitary $\nabla^b$-parallel frame, thus is a compact Lie group. For this $M'$, the deck transformation group $\pi_1(M')$ is a normal subgroup of the Lie group $\widetilde{M}=G\times {\mathbb R}^k$. As in the proof of Theorem 1, by passing to a finite cover of $M'$ if necessary, we may assume that $\Gamma_2\cong {\mathbb Z}^k$ and the deck transformation group is given by $\Gamma_{\rho}$ where  $\rho : {\mathbb Z}^k \rightarrow \Gamma_1$. The normality of $\Gamma_{\rho }$ in $G\times {\mathbb R}^k$ implies that $\Gamma_1$ is in the center of $G$, thus is finite. So when the Bismut holonomy group $h(\pi_1(M) )$ is finite, the map $\rho$ has finite image, which means that $M$ is covered by $G\times T^k$. Conversely, when $M$ is covered by $G\times T^k$, then both $\rho$ and $h$ has finite image of course. To summaries, we have the following

\begin{lemma}
Let $(M^n,g)$ be a compact Bismut flat manifold. Let $\widetilde{M}=G\times {\mathbb R}^k$ be its universal cover, where $G$ is compact semisimple. Let $h: \pi_1(M)\rightarrow U(n)$ be the Bismut holonomy map, and let $\rho : {\mathbb Z}^k \rightarrow I(G)$ be the homomorphism constructed in the proof of Theorem 1, namely, a subgroup of finite index in $\pi_1(M)$ which takes the form $\Gamma_{\rho}\cong {\mathbb Z}^k$ with elements
$$ \gamma_t (x,y) = (\rho (t) (x), y+t), \ \ \forall \  (x,y) \in G\times {\mathbb R}^k, \ \ \forall \ t \in {\mathbb Z}^k. $$
Then the following are equivalent:

(1). The image of $h$ is finite.

(2). The image of $\rho $ is finite.

(3). A finite unbranched cover of $M$ is a compact Lie group.

(4). A finite unbranched cover of $M$ is $G\times T^k$, where
$T^k={\mathbb R}^k/{\mathbb Z}^k$ is the torus.
\end{lemma}

\vsv

To illustrate the role of the deck transformation groups, let us
examine the isosceles  Hopf surface case. In this case, the
universal cover is the space ${\mathbb C}^2\setminus \{ 0\} = SU(2)
\times {\mathbb R}$, where the identification map is $ \phi (z) =
(A_z, \log |z|)$. Here $z=(z_1,z_2)$, $|z|^2=|z_1|^2+|z_2|^2$, and
$$ A_z = \frac{1}{|z|} \left[ \begin{array}{cc} z_1, & - \overline{z_2} \\ z_2, & \overline{z_1} \end{array} \right] \in SU(2).$$
The deck transformation group $\Gamma$ is a finite extension (by
unitary rotations) of the infinite cyclic group ${\mathbb Z}f$ where
$f(z_1, z_2) =(az_1, bz_2)$, with $0<|a|=|b|<1$. On $SU(2) \times
{\mathbb R}$, the action of the generator is $\gamma (A_z,y) = (\rho
(f) (A_z), \log |a| +y )$ where
$$ \rho(f) (A_z) = \frac{1}{ |a|\cdot  |z|}  \left[ \begin{array}{cc} az_1, & - \overline{bz_2} \\ bz_2, & \overline{az_1} \end{array} \right]  .$$
Note that $\rho(f)$ is always in the isometry group $I(G)$ of
$G=SU(2)$, but it will be in $G$ (as left multiplications) if and
only if $b=\overline{a}$. So in general, the image of $\rho$ is not
contained in $G$ itself. Also, the image of $\rho$ (or equivalently
the image of the holonomy map $h$) is finite if and only if both
$\frac{a}{|a|}$ and $\frac{b}{|a|}$ are roots of unity. So for a
generic choice of $|a|=|b|$, the primary isosceles Hopf surface
$({\mathbb C}^2\setminus \{ 0\} )/ {\mathbb Z}f$ does not have
finite Bismut holonomy, and the image of $\rho$ are not all left
multiplications of $G$.

\vsv

Next, let us give an example of a compact Bismut-flat threefold in Corollary 4, whose universal cover is $SU(2)\times {\mathbb R}^3= ({\mathbb C}^2\setminus \{ 0\} )\times  {\mathbb C}$, but none of the finite unbranched covers of $M$ can be the product of a Hopf surface and an elliptic curve.

\vsv

Let us consider the homomorphism $\rho : {\mathbb Z}^3 \rightarrow SU(2)$ defined by
$$ \rho (1,0,0)=A, \ \ \rho (0,1,0)=\cos \alpha I + \sin \alpha A, \ \ \rho (0,0,1) = \cos \beta I + \sin \beta A ,$$
where $\alpha$, $\beta$ are real numbers and
$$ A = \frac{1}{\sqrt{2}} \left[ \begin{array}{cc} i & 1 \\ -1 & -i \end{array} \right] $$
Note that $A^2=-I$, and $\rho (0,n,m) = \cos (n\alpha +m\beta ) I + \sin (n\alpha +m\beta ) A$ for any integer $n$ and $m$.  Let us take the values of $\alpha$ and $\beta$ so that $n\alpha + m\beta$ is not a rational multiple of $\pi$ for any $n$, $m\in {\mathbb Z}$. This would be the case if we take  $\alpha =\sqrt{2}\pi$ and $\beta =\sqrt{3}\pi$ for instance.

\vsv

Let us now consider the group $\Gamma_{\rho }\cong {\mathbb Z}^3$ which acts on $SU(2)\times {\mathbb R}^3$ by
$\gamma_t (x,y) = (\rho (t)(x), y+t)$ for any $(x,y)\in SU(2)\times {\mathbb R}^3$, where $t\in {\mathbb Z}^3$. Let $M^3_{\rho } = (SU(2)\times {\mathbb R}^3)/\Gamma_{\rho }$. Since both the metric and the complex structure on $SU(2)\times {\mathbb R}^3$ are left invariant, the elements of $\Gamma_{\rho }$ are holomorphic isometries, so $M_{\rho }$ is a compact Bismut flat threefold, and it is diffeomorphic to $SU(2)\times T^3$. However, for any subgroup $\Gamma ' \subseteq \Gamma_{\rho }$ with finite index, the (free part of the) abelian group $\rho (\Gamma ')$ still has rank $2$. Thus any finite unbranched cover of $M_{\rho }$ cannot be the product of a Hopf surface and an elliptic curve.

\vs

\vsv

\vsv

\vsv

\section{The non-compact case}

In this section, let us discuss non-compact Hermitian manifolds that are Bismut flat. It turns out that the compactness assumption in Theorem 1 can be dropped, thanks to a nice property about flat metric connections with skew-symmetric torsion on a Riemannian manifold, given by Agricola and Friedrich (\cite{AF}, Prop. 2.1). The result states that on a Riemannian manifold, if a metric connection $\nabla ''$ with skew-symmetric torsion is flat, then the torsion of $\nabla ''$ is parallel with respect to another metric connection $\nabla '= \frac{2}{3} \nabla + \frac{1}{3}\nabla ''$. In our notation, apply this result to the flat connection $\nabla^b$, we get the following:

\begin{lemma} {\bf (Agricola-Friedrich)} \label{lemma 13}
Let $(M^n,g)$ be a Hermitian manifold with flat Bismut connection $\nabla^b$. Then the torsion tensor $T^b$ of $\nabla^b$ is parallel with respect to the metric connection $\nabla ' = \frac{2}{3} \nabla + \frac{1}{3}\nabla^b$, where $\nabla$ is the Riemannian (Levi-Civita) connection.
\end{lemma}

By using formulae (32)-(35) in Lemma 9, one can also check directly that $\nabla ' T^b =0$. Now we are ready to prove Theorem 5.

\vs

\begin{proof} [\textbf{Proof of Theorem 5.}] Let $(M^n,g)$ be a simply-connected Bismut flat manifold. Let $e$ be a local unitary $\nabla^b$-parallel frame.  Under such an $e$, we have
\begin{eqnarray*}
T^b(e_i, e_j ) & = & -2\sum_k T_{ij}^k e_k \\
T^b(\overline{e_i}, e_j) & =  & 2\sum_k ( \overline{T_{ik}^j} e_k - T^i_{jk} \overline{e_k} )
\end{eqnarray*}
So the square norm $|T^b|^2 = 24\sum_{i,j,k} |T^k_{ij}|^2 = 24
|T|^2$. Since $T^b$ is $\nabla'$-parallel by Lemma 13, we know that
the square norm $|T^b|^2$ is a constant on $M$. Thus the left hand
side of formula (38) in Lemma 11 is identically zero, which implies
that $T^i_{kl,\overline{j}} =0$ for any indices. So under any local
unitary $\nabla^b$-parallel frame $e$, the components $T^k_{ij}$ of
the torsion of the Chern connection are all constants. By the proof
of Theorem 1, we know that $M$ is an open subset of a Samelson
space, and Theorem 5 is proved.
\end{proof}

\vs

Theorem 5 suggests that the flatness of the Bismut connection is perhaps more restrictive than the flatness of some other metric connections on a Hermitian manifold. For instance, in \cite{Boothby}, Boothby pointed out that in the non-compact case, a Chern flat metric doesn't have to have parallel Chern torsion, even in complex dimension $2$. Below let us give another example in complex dimension $2$.

\vsv

Let $\Omega \subseteq {\mathbb C}^2$ be a domain, and $f$, $h$ be holomorphic functions on $\Omega$, and $(z_1, z_2)$ be the standard coordinates of ${\mathbb C}^2$. Consider the Hermitian metric $g$ on $\Omega$ given by
$$\omega_g = \sqrt{-1} (e^{f+\overline{f}}dz_1\wedge d\overline{z}_1 +  e^{h+\overline{h}}dz_2\wedge d\overline{z}_2). $$
The Chern connection of $(\Omega , g)$ is flat, since it has a holomorphic unitary frame $e_1=e^{-f}\frac{\partial }{\partial z_1}$ and $e_2=e^{-h}\frac{\partial }{\partial z_2}$. It is K\"ahler if and only if both $\frac{\partial f}{\partial z_2}$ and $\frac{\partial h}{\partial z_1}$ are identically zero, and the components of the torsion of the Chern connection under the frame $e$ are
$$ T^1_{12} = -\frac{1}{2} \ \frac{\partial f}{\partial z_2} \ e^{-h} , \ \ \ \ \ T^2_{12} = \frac{1}{2} \ \frac{\partial h}{\partial z_1} \ e^{-f} .$$
For generic choices of $f$ and $h$, clearly the Chern torsion does not have constant norm, thus can not be parallel under any metric connection.

\vsv

There are also complete examples of this kind. For instance, consider the Chern flat Hermitian metric $g$ on ${\mathbb C}^2$ given by
$$ \omega_g = \sqrt{-1} ( \varphi_1\wedge \overline{\varphi_1} + \varphi_2 \wedge \overline{\varphi_2} ) , $$
where $\varphi_1=dx$, $\varphi_2=dy-2xydx$, and $(x,y)$ is the standard coordinate of ${\mathbb C}^2$. The torsion components under the unitary coframe $\varphi$ are $T^1_{12}=0$, $T^2_{12}=x$. So the norm of the Chern torsion $|T^c|^2= 16|x|^2$ is not a constant.

\vsv

To see that $g$ is complete, let $\sigma : [0,\infty ) \rightarrow {\mathbb C}^2$ be a smooth curve that goes to infinity. Write $\sigma (t) = (x(t), y(t))$.  Its length under $g$ is
$$ L_g(\sigma )= \int_0^{\infty } \sqrt{ |x'|^2 + |y'-2xyx'|^2 } dt . $$
Assume that $L_g(\sigma ) <\infty$. Then $\int |x'|dt < \infty $, so $|x(t)|< C$ for some constant $C$. Let $z(t)=e^{- x^2(t)}y(t)$, then
$$ | z' | = |e^{- x^2} (y' - 2xy x')| \leq e^{C^2} |y'-2xyx'|, $$
whose integral over $[0,\infty )$ is finite. So $z(t)$ stays bounded, which implies that $y(t)$ also stays bounded, as $|y|\leq e^{C^2}|z|$. But this is impossible as $(x(t), y(t))$ needs to go to infinity when $t\rightarrow \infty$. So $L_g(\sigma )$ must be $\infty$, and this shows the completeness of the metric $g$.

\vs

For Hermitian surfaces $(M^2,g)$ with flat Riemannian connection, there are also lots of non-compact examples, but there are no complete ones. In fact, if
we assume that $(M^2,g)$ is a complete Hermitian manifold with flat
Riemannian connection, then its universal cover $\widetilde{M}$ as a
Riemannian manifold is just the flat Euclidean space ${\mathbb
R}^4$. In \cite{SV} (Theorem 1.3), Salamon and Viaclovsky showed
that any orthogonal complex structure on ${\mathbb R}^4$ (or
${\mathbb R}^4$  deleting a subset with zero $1$-dimensional
Hausdorff measure) must be the standard one, namely,
$\widetilde{M}$ is holomorphically isometric to the flat ${\mathbb
C}^2$. In contrast, Borisov, Salamon, and
Viaclovsky in \cite{BSV} were able to construct infinitely many
nonstandard orthogonal complex structures on the Euclidean space $\mathbb{R}^6$.

\vsv

It is well-known that Hermitian surfaces $(M^2,g)$ with flat Riemannian connection correspond to holomorphic maps from $M^2$ into the space $Z$ of all almost complex structures on $\mathbb{R}^4$ compatible with the Euclidean metric and (a fixed) orientation. However, it is not necessarily easy to write down   such metrics explicitly in terms of the complex Euclidean coordinate $(z_1, z_2)$. Here we observe that such surfaces are locally determined by three holomorphic functions, and using this characterization, we can write down lots of explicit examples of such metrics.

\vsv

Let $U$ be a complex manifold of complex dimension $2$, and $u$, $v$, $f$ are holomorphic functions in $U$. Let us denote by
\begin{eqnarray}
\varphi_1 & = & \frac{1}{ \sqrt{2} \sqrt{\lambda } }  du +   \frac{1}{ \sqrt{2} \lambda \sqrt{\lambda }}  (i\overline{v}-u\overline{f})df \\
\varphi_2 & = & \frac{1}{ \sqrt{2} \sqrt{\lambda } }  dv -  \frac{1}{ \sqrt{2} \lambda \sqrt{\lambda }}  (i\overline{u}+v\overline{f})df
\end{eqnarray}
where $\lambda = 1+ |f|^2$. We have the following:

\begin{lemma}
Let $(M^2,g)$ be a Hermitian surface of flat Riemannian connection. Then for any $p\in M$, there exists a neighborhood $p\in U\subseteq M$ and three holomorphic functions $u$, $v$, $f$ in $U$, such that $\{ \varphi_1, \varphi_2\} $ given in $(40)$, $(41)$ forms a unitary coframe in $U$. Conversely, given any three holomorphic functions $u$, $v$, $f$ in a complex surface $U$ such that $\varphi_1\wedge \varphi_2$ is nowhere zero, the Hermitian metric $g$ using $\varphi$ as unitary coframe has flat Riemannian connection.
\end{lemma}

\begin{proof}
Let $(x_1, \ldots , x_4)$ be the standard coordinate of ${\mathbb R}^4$ and write $\epsilon = \ ^t\! ( \frac{\partial}{\partial x_1}, \ldots , \frac{\partial}{\partial x_4} )$.  Under the natural frame $\epsilon$, the elements $J$ in $Z$ are represented by matrices
\begin{equation}
 J_z=\left[ \begin{array} {cc} aE & bE-cI \\ bE+cI & -aE \end{array} \right] , \ \ \ \ \mbox{where} \ \ E= \left[ \begin{array} {cc} 0 & 1 \\ -1 & 0 \end{array} \right] ,
 \end{equation}
$z=x+iy \in {\mathbb C}\cup \{ \infty \}$, and  $(a,b,c)=(\frac{2x}{|z|^2+1}, \frac{2y}{|z|^2+1}, \frac{|z|^2-1}{|z|^2+1})$, which  identifies $Z\cong S^2$ with ${\mathbb P}^1={\mathbb C}\cup \{ \infty \} $. Namely, we have $J(\epsilon ) = J_z \epsilon$.

\vsv

In order to get an explicit expression of a local unitary frame, we look for a local  orthonormal frame $\tilde{\epsilon}=P\epsilon$ such that $J(\tilde{\epsilon }) = J_0\tilde{\epsilon}$, or equivalently, $P^{-1}J_0P =J_z$. While $P$ is highly non-unique, the following symmetric matrix
\begin{equation*}
 P= \frac{1}{  \sqrt{|z|^2+1} }  \left[ \begin{array} {cc} xI & yI-E \\ yI+E & -xI \end{array} \right]
 \end{equation*}
is clearly orthogonal and satisfies the condition $P^{-1}J_0P =J_z$.

\vsv

Write
$$ \epsilon' = \left[ \begin{array} {c} \epsilon_1 \\ \epsilon_2 \end{array} \right] , \ \ \epsilon'' = \left[ \begin{array} {c} \epsilon_3 \\ \epsilon_4 \end{array} \right] , \ \  \tilde{\epsilon}' = \left[ \begin{array} {c} \tilde{\epsilon}_1 \\ \tilde{\epsilon}_2 \end{array} \right] , \ \ \tilde{\epsilon}'' = \left[ \begin{array} {c} \tilde{\epsilon}_3 \\ \tilde{\epsilon}_4 \end{array} \right] , $$
then we have
$$ \tilde{\epsilon}' = \frac{1}{  \sqrt{|z|^2+1} } (x\epsilon' +(yI-E)\epsilon''), \ \ \ \tilde{\epsilon}'' = \frac{1}{  \sqrt{|z|^2+1} } ((yI+E)\epsilon' -x \epsilon''). $$
From this, we can form a local unitary frame  $e=\ ^t\!(e_1, e_2)$ by
$$e = \frac{1}{ \sqrt{2} } (\tilde{\epsilon}' - i \tilde{\epsilon}'' ) =  \frac{1}{ \sqrt{2} \sqrt{|z|^2+1} } \{   ( \overline{z}I-iE) \epsilon' + (i\overline{z} I - E) \epsilon'' \}  $$
Let $\{ \varphi_1, \varphi_2\}$ be the local unitary coframe on $M$ dual to $e$, then we have
\begin{equation}
 \left[ \begin{array}{ll} \varphi_1 \\ \varphi_2 \end{array} \right] = \frac{1}{ \sqrt{2} \sqrt{|z|^2+1} } \{ (zI+iE) \left[ \begin{array}{ll} dx_1 \\ dx_2 \end{array} \right] - (izI+E) \left[ \begin{array}{ll} dx_3 \\ dx_4 \end{array} \right] \} .
 \end{equation}
 Now if $(M^2,g)$ is a Hermitian surface with flat Riemannian connection. Fix any $p\in M$, we can choose a small neighborhood $U$ of $p$ and a local coordinate $(x_1, \ldots , x_4)$ centered at $p$, such that the natural frame $\epsilon = \frac{\partial }{\partial x}$ is orthonormal and parallel under the Riemannian connection. The complex structure on $M$ gives a smooth map $f$ from $U$ into $Z\cong {\mathbb P}^1$, such that the almost complex structure of $M$ at $q\in U$ corresponds to $J_{f(q)} \in Z$. As is well-known, the integrability of $J$ is equivalent to the holomorphicity of $f$.

\vsv

From the formula $(43)$ above, we get a local unitary coframe $\varphi$ in $U$:
\begin{equation}
 \left[ \begin{array}{ll} \varphi_1 \\ \varphi_2 \end{array} \right] = \frac{1}{ \sqrt{2} \sqrt{|f|^2+1} } \{ (fI+iE) \left[ \begin{array}{ll} dx_1 \\ dx_2 \end{array} \right] - (ifI+E) \left[ \begin{array}{ll} dx_3 \\ dx_4 \end{array} \right] \}  .
 \end{equation}
Write $\lambda = 1+|f|^2$, $t_1 = x_1+ix_3$, $t_2=x_2+ix_4$, then the above formula can be rewritten as
\begin{eqnarray}
\varphi_1 & = & \frac{1}{ \sqrt{2} \sqrt{\lambda } } (fd\overline{t}_1 +idt_2) \\
\varphi_2 & = & \frac{1}{ \sqrt{2} \sqrt{\lambda } }  (fd\overline{t}_2 -idt_1)
\end{eqnarray}
Since the $(0,1)$-components of $\varphi_1$, $\varphi_2$ are zero, we know that $u=f\overline{t}_1 +it_2$ and $v=f\overline{t}_2-it_1$ are both holomorphic functions. Expressing $\overline{t}_1$, $t_2$ in terms of $u$ and $\overline{v}$, we get
\begin{equation}
 t_1 = \frac{1}{\lambda } (f\overline{u} + iv) , \ \ \ \  t_2 = \frac{1}{\lambda } (f\overline{v} -iu).
 \end{equation}
Plugging them into (45) and (46), we get the expressions (40) and
(41). This proved the first part of the lemma.

\vsv

Conversely, if we start with three holomorphic functions $u$, $v$, $f$ in
$U$ with $\varphi_1\wedge \varphi_2$ nowhere zero, then the Hermitian
metric $g$ with metric form
$$ \omega = i (\varphi_1 \wedge \overline{\varphi}_1 +
\varphi_2 \wedge \overline{\varphi}_2 ) $$ will have flat Riemannian
connection. This is because if we define $t_1$ and $t_2$ by (47),
and let $x_1$ and $x_3$ (respectively $x_2$ and $x_4$) be the real
and imaginary parts of $t_1$ or $t_2$, then the formula (40) and
(41) becomes (45) and (46), and then $(44)$. From this, it is easy
to compute that
 the matrices of the Riemannian connection are
 $ \theta_1 =  \alpha I $ and $\theta_2 =  \beta E$, where
$$ \alpha = \frac{1}{2\lambda } ( fd\overline{f} -
\overline{f}df), \ \ \beta = - \frac{idf}{\lambda} . $$
Clearly, $\overline{\alpha } = -\alpha$, and
$ d\alpha = \beta \overline{\beta}$, $d\beta = 2 \beta \alpha$.
This means $\Theta_1=\Theta_2=0$, so the Riemannian connection of $\omega$
 is everywhere flat. This completes the proof of the lemma.
\end{proof}

Note that $\beta\overline{\beta} = \partial \overline{\partial } \log (1+|f|^2)$ is globally defined, so $\log (1+|f|^2)$ is defined up to an additive pluriharmonic function, but $f$ itself is not globally defined.

\vsv

Using Lemma 14, we can easily produce lots of explicit examples of Hermitian metrics with flat Riemannian connections. For instance, if we take $u=z_1$, $v=0$, $f=z_2$,  we get a Hermitian metric $g_1$ on ${\mathbb C}^{\ast }\times {\mathbb C}$:
$$ \omega_{g_1} = \frac{\sqrt{-1}}{(1+|z_2|^2)^2} \{ (1+|z_2|^2)dz_1 \wedge d\overline{z}_1 + |z_1|^2 dz_2 \wedge d\overline{z}_2 - z_1\overline{z}_2 dz_1 \wedge d\overline{z}_2 - z_2\overline{z}_1 dz_2 \wedge d\overline{z}_1 \} $$
If we let $u=z_1$, $v=z_2$ and $f=\sqrt{-1}z_1z_2$, then we get a Hermitian metric $g_2$ on ${\mathbb C}\times \Omega$, where $\Omega \subseteq {\mathbb C}$ is any domain not intersecting the unit circle $|z_2|=1$, by
$$ \omega_{g_2} = \frac{\sqrt{-1}}{(1+|z_1z_2|^2)^2} \{ (1-|z_2|^2)^2dz_1 \wedge d\overline{z}_1 + (1+|z_1|^2)^2 dz_2 \wedge d\overline{z}_2  \} .$$
We will leave it to the readers to verify that the square norm $|T^c|^2$ of the Chern torsion tensor for  $g_1$ or $g_2$ is not a constant. Note that for a given explicit metric such as $g_2$, it is a rather tedious task to compute its Riemannian curvature, without knowing a convenient unitary coframe a priori.

\vs

\noindent\textbf{Acknowledgments.} We would like to thank Bennett
Chow, Bo Guan, Gabriel Khan, Kefeng Liu, Lei Ni,  Valentino Tosatti, Damin Wu, Hongwei Xu, and Xiaokui Yang for their
interests and encouragement.

\vs

\end{document}